\documentclass[11pt,oneside]{amsart}
\usepackage[utf8]{inputenc}
\usepackage{amsmath,amssymb,mathtools}
\usepackage{csquotes}
\usepackage[linktocpage=true]{hyperref}
\usepackage{pb-diagram}
\usepackage{xcolor}
\usepackage{standalone}
\usepackage{tikz}

\usepackage{import}

\usepackage{xpatch}
\makeatletter   
\xpatchcmd{\@tocline}
{\hfil\hbox to\@pnumwidth{\@tocpagenum{#7}}\par}
{\ifnum#1<0\hfill\else\dotfill\fi\hbox to\@pnumwidth{\@tocpagenum{#7}}\par}
{}{}
\makeatother 

\theoremstyle{plain}
\newtheorem{theorem}{Theorem}[section]
\newtheorem{thm}[theorem]{Theorem}
\newtheorem{lemma}[theorem]{Lemma}
\newtheorem{corollary}[theorem]{Corollary}

\newtheorem{proposition}[theorem]{Proposition}

\newtheorem{prop-defn}[theorem]{Proposition-Definition}
\newtheorem{claim}{Claim}
\newtheorem*{claim*}{Claim}
\newtheorem{subclaim}{Subclaim}
\newtheorem*{theorem:main}{Main Theorem}

\newtheorem{thmint}{Theorem}

\theoremstyle{definition} 
\newtheorem{defn}[theorem]{Definition}
\newtheorem{definition}[theorem]{Definition}
\newtheorem{remark}[theorem]{Remark}
\newtheorem{question}[theorem]{Question}
\newtheorem{problem}[theorem]{Problem}
\newtheorem*{remark*}{Remark}
\newtheorem*{remarks*}{Remarks}

\newtheorem{example}[theorem]{Example}
\newtheorem{assumptions}[theorem]{Assumption}
\newtheorem{notation}[theorem]{Notation}
\newtheorem*{constants}{Constants}
\newtheorem*{constantsgauges}{Constants and Morse gauges}

\newcommand{\nest}{\sqsubseteq}
\usepackage{mathabx}
\usepackage{dsfont}
\newcommand{\propnest}{\sqsubsetneq}
\newcommand*{\pnest}{\propnest}
\newcommand{\orth}{\bot}
\newcommand{\transverse}{\pitchfork}
\newcommand*{\trans}{\transverse}
\newcommand{\gate}{\mathfrak g}
\newcommand*{\g}{\gate}
\newcommand*{\cal}{\mathcal}
\newcommand*{\C}{\mathcal{C}}

\newcommand*{\s}{\mathfrak{S}}
\newcommand*{\U}{\mathcal{U}}
\newcommand*{\V}{\mathcal{V}}
\newcommand*{\dist}{d}
\newcommand*{\eps}{\varepsilon}
\newcommand{\needaname}{\text{geometrically faithful\ }}
\DeclareMathOperator{\diam}{diam}

\DeclareMathOperator{\Isom}{Isom}
\DeclareMathOperator{\Rel}{Rel}

\DeclareMathOperator{\Stab}{Stab}

\usepackage{xcolor}
\definecolor{harrycomment}{rgb}{0.6,0,0.4}

\definecolor{markcomment}{rgb}{0.9,0,0.1}

\definecolor{antoinecomment}{rgb}{0.0,0.5,0.0}

\newenvironment{subproof}[1]{\smallskip\noindent\emph{#1.}}
            {\leavevmode\unskip\penalty9999\hbox{}\nobreak\hfill\quad\hbox{$\diamondsuit$}\par\medskip}

\title[Induced quasi-isometries]{Induced quasi-isometries of hyperbolic spaces, Markov chains, and acylindrical hyperbolicity}
\hypersetup{pdftitle = Induced quasi-isometries}

	\author[Antoine Goldsborough]{Antoine Goldsborough}
	\address{Maxwell Institute and Department of Mathematics, Heriot-Watt University, Edinburgh, UK}
	\email{ag2017@hw.ac.uk}

 \author{Mark Hagen}
 \address{School of Mathematics, University of Bristol, Bristol, UK}
 \email{markfhagen@posteo.net}

 \author{Harry Petyt}
 \address{Mathematical Institute, University of Oxford, UK}
 \email{petyt@maths.ox.ac.uk}
	
\author[Alessandro Sisto]{Alessandro Sisto}
	\address{Maxwell Institute and Department of Mathematics, Heriot-Watt University, Edinburgh, UK}
	\email{a.sisto@hw.ac.uk}

\address{Department of Mathematics \& Statistics, Swarthmore College, Swarthmore, PA, USA}
\email{jrussel2@swarthmore.edu}

\begin{document}

\maketitle
    \centerline{\textit{With an appendix by Jacob Russell}}

\begin{abstract}
    We show that quasi-isometries of (well-behaved) hierarchically hyperbolic groups descend to quasi-isometries of their maximal hyperbolic space. This has two applications, one relating to quasi-isometry invariance of acylindrical hyperbolicity, and the other a linear progress result for Markov chains.
    
    The appendix, by Jacob Russell, contains a partial converse under the (necessary) condition that the maximal hyperbolic space is one-ended.
\end{abstract}

\tableofcontents

\section{Introduction}

In this paper we consider groups $G$ acting on hyperbolic spaces $X$ such that every quasi-isometry of $G$ induces a quasi-isometry of $X$. We have two main motivations for considering this property, one relating to quasi-isometry invariance of acylindrical hyperbolicity, and the other to Markov chains. We discuss these separately below.

Examples of group actions with this induced quasi-isometry property include relatively hyperbolic groups whose peripheral subgroups are not relatively hyperbolic: the space $X$ is the coned-off Cayley graph. This can be deduced from \cite[Thm~4.1]{behrstockdrutumosher:thick}, which implies that quasi-isometries map peripheral subgroups into uniform neighbourhoods of peripheral subgroups. Results of this type originated in \cite{drutusapir:tree-graded}.

Our first result, essentially a special case of Theorem \ref{thm:qi-descend}, shows that many \emph{hierarchically hyperbolic groups} (HHGs) also fit into this framework. That is, their quasi-isometries descend to their \emph{maximal hyperbolic space}. We refer to the HHGs in question as \emph{well behaved} in this introduction---all naturally occurring HHGs are well behaved with the right choice of structure.  See Definition \ref{defn:well-behaved-HHS}.

\begin{thmint}\label{cor:HHG-QI-descend-intro}
Let $(G,\s)$ be a well-behaved HHG, with maximal hyperbolic space $\C S$, and let $\pi_S:G\to \C S$ be the associated projection. Every quasi-isometry $f:G\to G$ induces a quasi-isometry $\bar f:\C S\to \C S$ such that $\pi_S f$ and $\bar f\pi_S$ coarsely agree.
\end{thmint}
 
In Appendix \ref{sec:appendix}, Jacob Russell proves a partial converse to Theorem~\ref{cor:HHG-QI-descend-intro}. This is Theorem~\ref{thm:one-ende_converse}, which says that if $\C S$ is one-ended then all quasi-isometries of $\C S$ come from quasi-isometries of $G$. This generalises a result of Rafi--Schleimer \cite{rafischleimer_rigid}; see the appendix for more discussion.

Theorem~\ref{cor:HHG-QI-descend-intro} is a direct consequence of Lemma \ref{lem:making-good-behaviour} and Corollary \ref{cor:HHG-QI-descend-propaganda}, and
provides many examples of pairs $(G,X)$ where quasi-isometries of $G$ induce quasi-isometries of $X$. Indeed, the class of HHGs includes many: extensions and quotients of mapping class groups \cite{behrstockhagensisto:asymptotic,behrstockhagenmartinsisto:combinatorial,dowdalldurhamleiningersisto:extensions:2,russell:extensions}; cubical groups, such as special groups \cite{hagensusse:onhierarchical,chesser:stable}; 3--manifold groups \cite{hagenrussellsistospriano:equivariant}; Artin groups \cite{hagenmartinsisto:extra}; and combinations thereof \cite{berlairobbio:refined,berlynerussell:hierarchical} (in each case, either the HHG structure in the literature is already well-behaved, or it is easily modified to be so).
In the case of mapping class groups, the fact that quasi-isometries descend to the curve graph is a consequence of QI rigidity results from \cite{behrstockkleinerminskymosher:geometry,bowditch:large:mapping}, but Theorem \ref{cor:HHG-QI-descend-intro} is a significantly more direct route to the conclusion in this case.  For right-angled Artin groups, the fact that quasi-isometries descend to the contact graph of the universal cover of the Salvetti complex is in many cases a consequence of results of Huang \cite{huang:quasi:1,huang:quasi:2} showing that quasi-isometries descend to the \emph{extension graph}, along with an observation of Kim-Koberda relating this to the contact graph \cite{kimkoberda:geometry}, but it appears to be new for right-angled Artin groups where the outer automorphism group is sufficiently complicated.


\subsection*{Groups quasi-isometric to acylindrically hyperbolic groups}

It is arguably the main open question on acylindrically hyperbolic groups whether being acylindrically hyperbolic is a quasi-isometry--invariant property of groups, as asked, for example, in \cite[Problem 9.1]{dahmaniguirardelosin:hyperbolically} and \cite[Question 2.20]{osin:groups}. In fact, it is not even known whether acylindrical hyperbolicity is a commensurability invariant, and to highlight how little is known about this question, even the following are unanswered:
\begin{itemize}
    \item Let $G$ be acylindrically hyperbolic and let $H$ be of the form $G\rtimes \mathbb Z/2$. Is $H$ necessarily acylindrically hyperbolic?
\item Can a torsion group be quasi-isometric to an acylindrically hyperbolic group?
\item Let $G$ be a group quasi-isometric to $Out(F_n)$, with $n\geq 3$. Is $G$ necessarily acylindrically hyperbolic?
\end{itemize}

It is quite possible that acylindrical hyperbolicity is not a quasi-isometry, or even commensurability, invariant, making it quite interesting to obtain partial results in this direction. We isolate properties leading to such partial results in our set-up, where we consider pairs $(G,X)$ where quasi-isometries of the group $G$ induce quasi-isometries of the hyperbolic space $X$ being acted on. The existence of induced quasi-isometries in itself is not sufficient, as the hyperbolic space might be a point. However, it can be combined with some additional restrictions into a short list of assumptions from which proving the next theorem is a fairly routine matter. We state the assumptions on $(G,X)$ informally here, referring the reader to Definition \ref{defn:strong-21} for the precise version. 

\begin{enumerate}
    \item The action of $G$ on $X$ is nonelementary and acylindrical.
    \item Quasi-isometries of $G$ induce quasi-isometries of $X$.
\item (Morse detectability) A geodesic in $G$ is Morse if and only if it maps to a parametrised quasi-geodesic in $X$.
\end{enumerate}

We note that Morse detectability was abstracted in \cite{russellsprianotran:local} (inspired by \cite{abbottbehrstockdurham:largest,kentleininger:shadows}  among others) as a sufficient condition to show the Morse local-to-global property.

The following combines Corollary \ref{cor:acyl-qi} and Theorem \ref{thm:HHG-QI-acyl} (see Remark \ref{rem:HHG-weak-21}).

\begin{thmint}\label{thmint:acyl}
    Let $G$ be a group satisfying Definition \ref{defn:strong-21} (for example, let $G$ be a well-behaved HHG which is non-elementary and has unbounded maximal hyperbolic space). Then any group quasi-isometric to $G$ is acylindrically hyperbolic.
\end{thmint}

\subsection*{Application to Markov chains}
The study of Markov chains on groups as a ``quasi-isometry invariant'' generalisation of random walks was initiated in \cite{goldsboroughsisto:markov}. Given quasi-isometric groups, $G$ and $H$, and a simple random walk on $H$, there is no meaningful notion of a random walk on $G$ induced by the one on $H$. Markov chains resolve this issue for non-amenable groups since the push-forward (see Section \ref{sec:prelim} for definitions) of a tame Markov chain by a bijective quasi-isometry is again a tame Markov chain, and quasi-isometries between non-amenable groups are bounded distance from  bijective ones \cite{Whyte1999:AmenabilityBE}. 

For a group $G$ acting on a hyperbolic space $X$ with basepoint $x_0$, we say that a Markov chain (or random walk) $(w^{o}_n)_n$ in $G$ makes \textit{linear progress with exponential decay in $X$} if there is a constant $C>0$ such that for all $n$ and $o \in G$ we have $\mathbb P [ d_X(ox_0,w_n^{o}x_0) \geq n/C ] \geq Ce^{-n/C}. $ 
In the case of random walks, establishing this property was done in   \cite{mahertiozzo:random} for weakly hyperbolic groups. This property feeds into the proof of several results for random walks; such as a Central Limit Theorem for the random walk \cite{MathieuSisto}, genericity of loxodromic elements \cite{mahertiozzo:random} or that random subgroups of weakly hyperbolic groups are free \cite{taylortiozzo:random}, among many others. 

In a similar spirit, the main result of \cite{goldsboroughsisto:markov} was establishing that tame Markov chains in $G$ make \textit{linear progress} in the hyperbolic space $X$ for many groups $G$ acting on a hyperbolic space $X$. Examples of groups admitting such an action on a hyperbolic space include (non-elementary) relatively hyperbolic groups and acylindrically hyperbolic $3$-manifold groups. This enabled showing a Central Limit Theorem for random walks on groups quasi-isometric to such groups, and this property is used in \cite{goldsboroughSistorandom} to study random divergence (see below), More generally, linear progress is intended as a crucial starting point for further study of Markov chains.

One of the applications of Theorem \ref{cor:HHG-QI-descend-intro} is in establishing linear progress in the case where $G$ is a hierarchically hyperbolic group and $X$ is the maximal hyperbolic space in the HHG structure of $G$. We work with a more restrictive class of Markov chains than in \cite{goldsboroughsisto:markov}, as we require a property that we call \emph{quasi-homogeneity}, see Section \ref{sec:prelim}. Importantly, this property is satisfied by simple random walks and their push-forwards by bijective quasi-isometries. The following is Theorem \ref{thm:linear_progress} (see Remark~\ref{rem:HHG-weak-21} for the verification of Assumption~\ref{assumpt:space} for HHGs, part of which is Theorem~\ref{thm:qi-descend}). 

\begin{thmint}\label{cor:linear-progress-HHG}
Let $G$ be a group acting on a hyperbolic space $X$ and satisfying Assumption~\ref{assumpt:space} (for example, let $G$ be a well-behaved HHG which is non-elementary and has unbounded maximal hyperbolic space $\mathcal CS$). Then any tame, quasi-homogeneous Markov chain on $G$ makes linear progress with exponential decay in $X$.
\end{thmint}

One consequence is that, for groups $G$ as in the theorem, the random divergence defined in \cite{goldsboroughSistorandom} (and chosen according to tame, quasi-homogeneous Markov chains) is the same as the divergence of $G$. This means that generic points, chosen according to these Markov chains, realise the worst-case scenario for divergence, see \cite[Thm~1.1]{goldsboroughSistorandom}.
Note that if $G$ satisfies the stronger Assumption~\ref{defn:strong-21}, then using only Theorem~\ref{cor:linear-progress-HHG} and \cite[Thm~7.7]{goldsboroughsisto:markov} one can deduce a Central Limit Theorem on groups quasi-isometric to $G$. This also follows from Theorem~\ref{thmint:acyl} and \cite{MathieuSisto}.

\subsection*{Questions} Several questions arise. First, it would be interesting to find more classes of group actions on hyperbolic spaces such that quasi-isometries of the group descend to quasi-isometries of the space, in the sense of Theorem \ref{cor:HHG-QI-descend-intro}. For short, we will say in this case that \emph{quasi-isometries descend}.

\begin{problem}
Find more classes of groups admitting non-elementary actions on hyperbolic spaces with the property that quasi-isometries descend.
\end{problem}

To mention two specific instances:

\begin{question}
\begin{enumerate}
    \item[] 
    \item Do CAT(0) groups with rank-one elements admit actions on hyperbolic spaces such that quasi-isometries descend?\label{item:cat0-question}

    \item  \label{item:small-C-question} Do small-cancellation groups (of various flavours)?
\end{enumerate}
\end{question}

A candidate hyperbolic space for \eqref{item:cat0-question} could be the hyperbolic model from \cite{petytsprianozalloum:hyperbolic}, while for \eqref{item:small-C-question} (say for $C'(1/6)$ groups) it could be the space constructed in \cite{grubersisto:infinitely}.  For the CAT(0) question, any use of the model from \cite{petytsprianozalloum:hyperbolic} would have to use the geometric group action on the CAT(0) space in an essential way in view of the example in \cite{vest:curtain}. 

More strongly, it would be useful to identify more classes of groups such that Theorems \ref{thmint:acyl} and \ref{cor:linear-progress-HHG} apply. But we would like to emphasise that variations of said theorems should be possible. For instance, Theorem \ref{thmint:acyl} does not apply to $C'(1/6)$ groups for the candidate space given above as the action is not acylindrical in general, but a more general theorem might. Also, it is not known whether the action of a CAT(0) group on the hyperbolic model is acylindrical. A positive answer to either of the following questions might come from a more general version of Theorem \ref{thmint:acyl}. 

\begin{question}
\begin{enumerate}
    \item[]

    \item Let $G$ be a group quasi-isometric to an acylindrically hyperbolic CAT(0) group. Is $G$ necessarily acylindrically hyperbolic?

    \item Let $G$ be a group quasi-isometric to an infinitely presented $C'(1/6)$ group. Is $G$ necessarily acylindrically hyperbolic?
\end{enumerate}
\end{question}

Of course, one can also ask about analogues of Theorem \ref{cor:linear-progress-HHG}, whose conclusion we refer to as the linear progress for Markov chains property.

\begin{question}
\begin{enumerate}
    \item[]

    \item Do CAT(0) groups with rank-one elements admit non-elementary actions on hyperbolic spaces with the linear progress for Markov chains property? 

    \item Do small-cancellation groups (of various flavours)?
\end{enumerate}
\end{question}

Finally, motivated by the appendix, it is natural to ask:

\begin{question}
    Which HHGs have an HHG structure with unbounded products and one-ended $\nest$--maximal hyperbolic space?
\end{question}

Specifically, we believe that the answer is not known for extra-large type Artin groups and extensions of lattice Veech groups, for instance. In fact, there is even no known classification of right-angled Coxeter groups admitting an HHG structure as in the question.  However, for right-angled Artin groups, one can use \cite{abbottbehrstockdurham:largest} together with the HHG structure with unbounded products and maximal hyperbolic space a quasi-tree \cite{behrstockhagensisto:hierarchically:1}, to see that no HHG structure as in the question has one-ended maximal hyperbolic space.   

\subsection*{Outline of paper and proofs}
Section \ref{sec:prelim} contains some general geometric group theory preliminaries, as well as preliminaries on Markov chains. We also state the relevant assumptions on group actions and Markov chains that we will need later.

The hardest theorem in this paper is Theorem \ref{cor:linear-progress-HHG}, whose proof is contained in Sections \ref{sec:incompatible}--\ref{sec:proof}. In particular, in Sections \ref{sec:incompatible} and \ref{sec:bounded}, which contain the core of the geometric arguments involved, we will show roughly that with positive probability the Markov chain has a bounded projection to the axis of a fixed WPD element. This will be then used in Section \ref{sec:proof} to check a criterion for linear progress from \cite{goldsboroughsisto:markov}. 

The rough idea to show the bounded projection property is the following, and the reader might want to look at Figure \ref{fig:tikz:my} for reference. If the property fails, then with overwhelming probability the Markov chain creates a very large projection onto the axis. However, there is a positive probability that the Markov chains starts off in a Morse direction different from the axis. If it does, it needs to undo this second projection first, before creating the projection on the axis. We can also repeat this argument with another Morse direction which has a very different Morse gauge. Therefore, the Markov chain (from any basepoint due to quasi-homogeneity) has a large probability of creating a large projection in two directions with very different Morse gauges. This is not yet a contradiction, because we need to know that the two directions have different Morse gauges ``close to the basepoint'' rather than, say, starting out in the same way and then diverging later. This is what the notion of incompatible Morse rays from Section \ref{sec:incompatible} is supposed to capture, and in that section we study it and prove the required preliminary results.

In Section \ref{sec:descent} we show that quasi-isometries of well-behaved HHS descend to their maximal hyperbolic space, and related results. There are two main ideas here. The first one is to use a result from \cite{behrstockhagensisto:quasiflats} which says, roughly, that quasi-isometries between HHSs descend to quasi-isometries of ``simpler'' HHSs (certain so-called factored spaces) obtained coning-off certain product regions. One can repeat the procedure until the ``simpler'' HHSs are actually just hyperbolic spaces, but those will not in general be the maximal hyperbolic spaces of the HHSs as there might be further (quasiconvex) subspaces to cone-off to get there; this happens even for mapping class groups. The additional idea allows us to recognise these subspaces, and roughly we show in Proposition \ref{prop:product_characterisation} that two points are in the same subspace to be coned-off if and only if their coarse fibres in the original HHS are ``parallel''.

In Section \ref{sec:acyl-qi} we show our results related to quasi-isometry invariance of acylindrical hyperbolicity. The idea here is the following. If one has a group $G$ acting on a hyperbolic space $X$ with the property that quasi-isometries descend, and if $H$ is quasi-isometric to $G$, then $H$ quasi-acts on $G$ and therefore on $X$. The quasi-action on $X$ can be promoted to an action on a space $Y$ quasi-isometric to $X$, which is also hyperbolic.  This action admits a loxodromic element due to the classification of actions on hyperbolic spaces, and what is left to show is that any loxodromic is WPD. This comes from acylindricity of the original action, which can be translated into a geometric property about preimages of balls being geometrically separated as in \cite{sisto:quasi-convexity}. 

Finally, the appendix by Jacob Russell contains the result about quasi-isometries of the maximal hyperbolic space of an HHG coming from quasi-isometries of the HHG, under suitable conditions.

\subsection*{Acknowledgements}  We thank the organisers of the thematic programme on geometric group theory at the Centre de Recherches Math\'ematiques, where some of the work on this project was done. Goldsborough was supported by the EPSRC DTA studenship EP/V520044/1. Russell was supported by NSF grant  DMS-2103191. 
 The authors of the non-appendix part of the paper thank Jacob Russell for pointing out a subtlety involving the application of \cite[Thm. 3.7]{abbottbehrstockdurham:largest} and Abdul Zalloum for interesting conversations.  We are also grateful to the referee for numerous helpful comments.

\section{Background and assumptions}
\label{sec:prelim}

To set notation, we recall some standard notions from geometric group theory. All hyperbolic spaces considered in this paper will be geodesic.

\subsection{Geometric group theory notions and definitions} 
\label{subsec:ggt_defns}

A map $f:Y \to X $ between metric spaces is called a $(\lambda, \epsilon)$-\textit{quasi-isometric embedding}, with $\lambda \geq 1, \epsilon \geq 0$, if for all $x,y \in Y $ we have: $
\lambda^{-1}d_Y(x,y)-\epsilon \leq d_X(f(x),f(y)) \leq \lambda d_Y(x,y) +\epsilon$. We say that $f$ is a $(\lambda,\epsilon)$--quasi-isometry if, in addition, for all $x \in X$, there exists an element $y \in Y$ such that $d_X(f(y), x) \leq \epsilon$. If $Y$ is a segment of $\mathbb R$, we call the image of $f$ in $X$ a  $(\lambda, \epsilon)$-\textit{quasi-geodesic}.

We will call a  $(\lambda, \lambda)$-quasi-geodesic a $\lambda$-\textit{quasi-geodesic}, and similarly for quasi-isometric embeddings and quasi-isometries. A subset $Y$ of a geodesic metric space $X$ is \textit{quasi-convex} if there is a constant $C \geq 0$ such that all geodesics with endpoints in $Y$ stay within the $C$-neighbourhood of $Y$. Further, we say that a map $f:X \to Y$ between metric spaces is \textit{R-coarsely Lipschitz} if $d_Y(f(x), f(y)) \leq Rd_X(x,y)+R$ for all $x,y \in X$.

Let $M$ be a function $\mathbb [1, \infty) \times [0, \infty) \to [0, \infty)$. We say a (quasi)-geodesic $\gamma$ is $M$-\textit{Morse} if any $(\lambda, \epsilon)$-quasi-geodesic with endpoints on $\gamma$ stays within the $M(\lambda, \epsilon)$-neighbourhood of $\gamma$. We call $M$ the \textit{Morse gauge of} $\gamma$. We can always assume that $M$ takes values in $\mathbb N$. Note that a Morse quasi-geodesic is quasi-convex. 

If a group $G$ acts on a hyperbolic metric space $X$ (with basepoint $x_0$), we say an element $g$ is \textit{loxodromic} if the map $\mathbb Z \to X$ given by $n \to g^nx_0$ is a quasi-isometric embedding. In this case, $\langle g \rangle x_0$ is quasi-convex in $X$. We say that $g$ satisfies the \textit{weak proper discontinuity condition}, or that $g$ is \emph{WPD}, if for all $\kappa >0$ and $x_0 \in X$ there exists $N \in \mathbb{N}$ such that 
$$
\# \{ h \in G \vert \quad d_X(x_0, hx_0)< \kappa, \quad d_X(g^{N}x_0, hg^{N}x_0) < \kappa \} < \infty.
$$ 
Each loxodromic WPD element $g$ is contained in a unique maximal elementary subgroup of $G$, denoted $E(g)$ and called the \textit{elementary closure} of $g$, see \cite[Lemma~6.5]{dahmaniguirardelosin:hyperbolically}.

Finally, all groups we consider are finitely generated, and whenever we consider a group $G$ we automatically fix a word metric $d_G$ on $G$ coming from a finite generating set.

\subsection{Assumptions on the group action}
Now that we recalled the relevant notions, we can state our assumptions on the group action on a hyperbolic space that we will use for our result on linear progress.  The slightly different assumptions needed for the results on acylindrical hyperbolicity are postponed to Section \ref{sec:acyl-qi}.

Let $G$ be a group acting on a hyperbolic space $X$ and fix $x_0\in X$. We write $\rho:G\to X$ for the corresponding orbit map, though sometimes we shall suppress this. We shall use:

\begin{assumptions} 
\begin{enumerate} \label{assumpt:space}
\item[] 
\item \label{item:exist_wpd}    Some element of $G$ acts on $X$ as a loxodromic WPD.
\item \label{item:exist_quasi_isometry} \emph{Quasi-isometries of $G$ descend to $X$.} This means that, for each $\nu$ there exists $\lambda$ such that if $\phi:G\to G$ is a $\nu$--quasi-isometry, then there is some $\lambda$--quasi-isometry $\bar \phi:X\to X$ such that $\dist_X(\phi(g)x,\, \bar\phi(gx))\leq \lambda$ for all $g\in G$ and $x\in X$.
\item \label{item:morse_detectable} \emph{Partial (Morse) detectability.} This means: for every Morse gauge $M$ there exists $\lambda$ such that if $\gamma\subset G$ is an $M$--Morse geodesic, then $\rho\gamma\subset X$ is a $\lambda$--quasi-geodesic.
\end{enumerate}
\end{assumptions}

The notion of \emph{Morse detectability} of $G$ is \cite[Def.~4.17]{russellsprianotran:local}, which is stated as an equivalence of two properties, and the partial detectability hypothesis above is one of the two implications.

\subsection{Projections}\label{subsec:projections}

Given a group $G$ acting on a hyperbolic space $X$, we will make use of $X$ to define 'projection maps' to subsets of $G$. The following definition makes this precise.

\begin{definition}(\cite[Definition 3.2]{goldsboroughsisto:markov})
\label{defn:X_proj}
Let $G$ act on a hyperbolic space $X$, and fix $x_0\in X$. For $A\subseteq G$, an $X$-\textit{projection} is a retraction $\pi_A: G \rightarrow A$ such that for all $g\in G$ the point $\pi_A(g)x_0$ is a closest point in $Ax_0$ to $gx_0$. 
\end{definition}

Given a subset $B$ of a metric space $X$, we call a map $\pi:X\to B$ a \emph{closest-point projection} if $d_X(x,\pi(x))=d_X(x,B)$ for all $x\in X$. The following lemma is a well-known exercise in hyperbolic geometry, see e.g. \cite[Lemma 2.1]{goldsboroughsisto:markov}.

\begin{lemma} 
\label{lem:exo_hyperbolicspace}
Let $X$ be a $\delta$-hyperbolic space. Let $Q$ be a quasi-convex set and $\pi_Q:X \rightarrow Q$ a closest-point projection. There exists a constant $R>0$ depending only on $\delta$ and the quasi-convexity constant such that the following hold.
\begin{enumerate}
\item $\pi_Q$ is $R$-coarsely Lipschitz.
\item For all $x,y \in X$ with $d_{X}(\pi_Q(x), \pi_Q(y)) \geq R$ and for any geodesic $[x,y]$ from $x$ to $y$, there are points $m_1, m_2 \in [x,y]$ such that $d_X(m_1,\pi_Q(x)) \leq R$ and $d_X(m_2,\pi_Q(y)) \leq R$. Furthermore, the subgeodesic of $[x,y]$ from $m_1$ to $m_2$ lies in the $R$-neighbourhood of $Q$.
\item If $\pi'_Q:X\to Q$ is another closest-point projection, then $\dist_X(\pi_Q(x),\pi'_Q(x))\le R$ for all $x\in X$.
\end{enumerate}
\end{lemma}

A particular consequence of the final part of Lemma~\ref{lem:exo_hyperbolicspace} is that if a group $G$ acts on a hyperbolic space $X$, then whenever $A\subset G$ has quasiconvex orbit and $\pi_A$ is an $X$--projection, we have a uniform bound $\dist_X(\pi_A(p)x_0,\pi_{Ax_0}(px_0))\le R$.

The following lemma is well known. Inequalities of this type are often referred to as ``Behrstock inequalities''.

\begin{lemma} 
\label{lem:behrtock_inequality}
Let $X$ be a $\delta$-hyperbolic space and let $Q_1 \neq Q_2$ be two quasi-convex sets and $\pi_{Q_i} :X \to Q_i$ be closest point projections. There is a constant $B$ only depending on $\delta$ and the quasi-convexity constants of $Q_1$ and $Q_2$ such that for all $x \in X$ we have 
$$ d_X\left(\pi_{Q_1}(x), \pi_{Q_1}(Q_2) \right) > B \implies d_X\left ( \pi_{Q_2}(x), \pi_{Q_2}(Q_1) \right) \leq B.$$
\end{lemma}

The following lemma is an exercise in hyperbolic geometry and it states that in a hyperbolic space quasi-isometries and closest-point projections are coarsely compatible. 

\begin{lemma} 
\label{lem:quasi_isometries_and_projections}
Let $X$ be a $\delta$-hyperbolic space and $f:X \to X$ a $\lambda$-quasi-isometry. Let $Q \subseteq X$ be a quasi-convex subspace and let $\pi_Q:X \to Q$ be a closest point projection. Then there exists a constant $N$, depending only on $\lambda$, $\delta$, and the quasiconvexity constant, such that for all $x \in X$,
$$ d_X\left(\pi_{f(Q)}(f(x)),f(\pi_Q(x)\right) \leq N.  $$
\end{lemma}

\begin{notation}
\label{not:projections}
For a group $G$ acting on a hyperbolic space with fixed basepoint $x_0$ and corresponding orbit map $\rho$, given a subset $\alpha\subseteq G$ (usually a quasi-geodesic), we will always implicitly fix an $X$-projection $\pi_\alpha: G\to\alpha$. We also abbreviate the diameter of the union of two projections as measured in $X$ by
$$d_{\alpha}(x,y) \,=\, \diam_X\Big(\rho\big(\pi_\alpha(x)\cup \pi_\alpha(y)\big)\Big),$$
where $x$ and $y$ can be either points or subsets of $G$. 
\end{notation}

\subsection{Markov chains}

We refer the reader to \cite{goldsboroughsisto:markov} for more background information on Markov chains, while here we describe the notion informally. A Markov chain on a group arises when transition probabilities $p(g,h)$ are assigned for all $g,h\in G$. These encode the probability that the Markov chain ``jumps'' from $g$ to $h$ in one step. The probability of going from $g$ to $h$ in $n$ steps is a sum over all possible trajectories, that is, sequences of jumps, of length $n$ of executing that exact sequence of jumps, and this is the product of the relevant transition probabilities. 

We usually denote a Markov chain on a group by $(w_n^*)$, where $w_n^o$ denotes the position of the Markov chain starting at $o$ after $n$ steps, and we are usually interested in quantities such as $\mathbb P[w_n^o=g]$, the probability of getting from $o$ to $g$ in $n$ steps.

The following notion of tameness was defined in \cite{goldsboroughsisto:markov}.

\begin{definition}[Tame]\label{defn:tame}
A Markov chain on $G$ is \textit{tame} if it satisfies the following:
\begin{enumerate}
\item\label{item:bounded_jumps} \emph{Bounded jumps:} There exists a finite set $S\subseteq G$ such that $\mathbb P[w^g_1=h]=0$ if $h\notin gS$.
\item\label{item:non-amen} \emph{Non-amenability:} There exist $A>0$ and $\rho<1$ such that for all $g,h\in G$ and $n\geq 0$ we have $\mathbb P[w^g_n=h]\leq A\rho^n$.
\item\label{item:irred} \emph{Irreducibility:} For each $s \in G$ there exist constants $\epsilon_s, K_s>0$ such that for all $g \in G$ we have $\mathbb P[w^g_k=gs] \geq \epsilon_s$ for some $k \leq K_s$.
\end{enumerate}
\end{definition}

For a bijection $\phi:G \to H$ and a Markov chain $(w_n^{*})$ on $G$, there is a natural push-forward Markov chain on $H$, which we denote by $\phi_{\#}(w_n^{*})$. This is the Markov chain such that $\mathbb P \big[\phi_{\#}(w_n^{o})=h \big]\,=\,\mathbb P\big[w_n^{o}=\phi^{-1}(h) \big]$ for all $h,o\in H$ and $n\geq 0$. To clarify, $\phi_{\#}(w_n^{o})$ is the instance of the Markov chain starting at $\phi(o)$.

\begin{definition}[Quasi-homogeneous]
\label{assumpt:markov_chain_quasihomogeneous}
A Markov chain $(w_n^{*})$ is \emph{quasi-homogeneous} if it has the following property for some $\nu$. For every $p,q \in G$ there is a bijective $\nu$--quasi-isometry $\phi:G \rightarrow G$ with $\phi(p)=q$ and $\phi_\#(w_n^o)=(w_n^{\phi(o)})$ for all $o\in G$.

\end{definition}

\begin{remark}
A random walk driven by a measure whose support is bounded and generates the group as a semi-group is a tame Markov chain, and moreover any push-forward of such a random walk by a bijective quasi-isometry is a tame Markov chain by \cite[Lemma 2.8]{goldsboroughsisto:markov}. Since any random walk is group-invariant, it is readily seen that such a push-forward is in fact also quasi-homogeneous.
\end{remark}

\section{Incompatible Morse rays}
\label{sec:incompatible}

In this section we consider a group $G$ acting on a hyperbolic space $X$ such that Assumption \ref{assumpt:space}.\eqref{item:morse_detectable} holds, that is to say, Morse geodesics in $G$ map to parametrised quasi-geodesics in $X$.

In particular, we are interested in a criterion to guarantee that two Morse rays in $G$ travel in genuinely distinct directions in $X$. The key definition to achieve this is the following, which roughly describes a ray which, while being Morse, has an initial subgeodesic of controlled length which is ``not very Morse''. 

\begin{definition}
Let $Z$ be a metric space, let $M: [1,\infty)\times[0,\infty) \to \mathbb N $ be a Morse gauge, and let $\kappa, L \geq 0$. We say a Morse ray $\beta: [0, \infty)\to Z$ is \textit{$(M, \kappa, L)$-incompatible} if there exists a $(k,c)$-quasi-geodesic $\mu$ with endpoints on $\beta\vert_{[0, L]}$, such that there is a point $p \in \mu$ with $d(p, \beta) > M(k,c+2\kappa)+2\kappa$.
\end{definition}

The main result in this section 
is Lemma \ref{lem:existence_incompatible}, guaranteeing the existence of incompatible rays. We need two preliminary results. The first one says that, given a Morse ray $\gamma$ and a quasi-geodesic ray $\alpha$, either $\alpha$ is contained in a controlled neighbourhood of $\gamma$, or $\alpha$ at some point starts diverging from $\gamma$ at a linear rate.

\begin{lemma}[\cite{cashenmackay:metrizable}, Corollary 4.3]
\label{lem:cashen_mackay}
Let $\alpha$ be a $M$-Morse quasigeodesic ray and let $\beta$ be an $(a,b)$-quasi-geodesic ray, both in some proper geodesic metric space $Z$. There exist $\kappa_1 = \kappa_1 (M, a,b)$ and $\kappa_2=\kappa_2(\kappa_1,a,b)$ such that the following holds.  If $d(\alpha(0), \beta) \leq \kappa_1$ then there are two possibilities: 
\begin{itemize}
	\item 
            $\beta$ is contained in the $\kappa_2$-neighbourhood of $\alpha$.
    \item There exists a last time $T_0$ such that 
            for all $t\in[0,\infty)$ we have $$d(\beta(t), \alpha) \geq \frac{1}{2a}(t-T_0)-2(b+\kappa_1). $$
	\end{itemize}
\end{lemma}

We note that in \cite{cashenmackay:metrizable} $\alpha$ is a geodesic ray, but since our $\alpha$ is Morse it lies within finite Hausdorff distance of a geodesic ray.


The following result says that a geodesic triangle where two sides are Morse is thin.

\begin{lemma}\cite[Lemma 2.2]{cordes:morse}
\label{lem:cordes_morse_triangles_thin}
For all Morse gauges $M_1, M_2$ there is a constant $\Delta=\Delta(M_1, M_2)$ such that every geodesic triangle where two of the sides are respectively $M_1$- and $M_2$-Morse is $\Delta$-thin. 
\end{lemma}

The following lemma states that an $M$-Morse ray and an $M$-incompatible Morse ray have bounded projection onto each other. The constant dependencies are rather involved, but they are crucial for the proof of the key Proposition \ref{prop:small_proj}. Recall the notation that the orbit map $\rho:G\to X$ given by $g\mapsto gx_0$ is $K$--Lipschitz.

\begin{lemma}
\label{lem:small_gromov}
Suppose that $G$ acts on a $\delta$--hyperbolic space $X$ and that Assumption \ref{assumpt:space}.\ref{item:morse_detectable} holds. For every Morse gauge $M$ and constant $\nu$ there exists $\kappa=\kappa(M,\nu)$ such that for every Morse gauge $M'$ and constant $L$ there exists $D=D(M,\nu,M',L)$ such that the following holds.

Let $\alpha\subset G$ be an $M$-Morse $\nu$-quasi-geodesic ray issuing from $1\in G$. If $\beta\subset G$ is an $M'$-Morse ray issuing from $1\in G$ that is $(M,\kappa,L)$-incompatible, then
$$\diam_G\left(\pi_{\alpha}(\beta)\right) \leq D \quad \text{and} \quad 
    \diam_G\left(\pi_{\beta}(\alpha)\right) \leq D.$$
\end{lemma}

\begin{proof}
Let $\kappa_1$ and $\kappa_2$ be the constants given by applying Lemma~\ref{lem:cashen_mackay} to the $M$--Morse quasigeodesic $\alpha$ and the geodesic $\beta$. According to Assumption~\ref{assumpt:space}.\eqref{item:morse_detectable}, there are $\lambda_\alpha=\lambda_\alpha(M)$ and $\lambda_\beta=\lambda_\beta(M')$ such that $\rho\alpha\subset X$ is a $\lambda_\alpha$--quasigeodesic and $\rho\beta\subset X$ is a $\lambda_\beta$--quasigeodesic. Let $R=R(\lambda_\alpha)$ be given by applying Lemma~\ref{lem:exo_hyperbolicspace} to $\rho\alpha\subset X$. and let $\Delta=\Delta(M,M')$ be the constant from Lemma \ref{lem:cordes_morse_triangles_thin}. 

We first bound $\diam_G(\pi_\alpha(\beta))$. For this, we shall consider sufficiently large constants
\[
\eps\,=\,\eps(\lambda_\alpha,\lambda_\beta,\delta) \quad\text{and}\quad U\,=\,U(M,M',R,\eps,\Delta).
\]
From these, we define 
\[
D_1 \,=\, (U+L+R+\eps+2\Delta+2M(1,0)+6\kappa_1)K\lambda_\alpha+\lambda_\alpha^2.
\]

Suppose that, contrary to the desired result, there is some $v\in\beta$ such that $\dist_G(1,\pi_\alpha(v))>D_1$. Since $\rho\alpha$ is a $\lambda_\alpha$--quasigeodesic, we have $\dist_X(x_0,\pi_\alpha(v)x_0)>K(U+L+R+\eps)$. As noted after Lemma~\ref{lem:exo_hyperbolicspace}, $\pi_\alpha(v)x_0$ differs from $\pi_{\rho\alpha}(vx_0)$ by at most $R$. 

Since $\rho\beta$ is a $\lambda_\beta$--quasigeodesic, the Morse lemma implies that every geodesic from $x_0$ to $vx_0$ stays uniformly close to $\rho\beta$, so Lemma~\ref{lem:exo_hyperbolicspace} provides a point $u\in\beta$ such that $\dist(ux_0,\pi_{\rho\alpha}(vx_0))\le\eps$, where $\eps=\eps(\lambda_\alpha,\lambda_\beta,\delta)$ is a uniform constant. Let us write $u=\beta(t)$. By the construction of $u$, we have 
\begin{align*}
t-U \,=\, \dist_G&(1,u)-U \,\ge\, \frac1K\dist_X(x_0,ux_0)-U \\
&\ge\, \frac1K\big(\dist_X(x_0,\pi_\alpha(v)x_0) \,-\, \dist_X(\pi_\alpha(v)x_0,\pi_{\rho\alpha}(vx_0)) 
    \,-\, \dist_X(\pi_{\rho\alpha}(vx_0),ux_0)\big) \,-\, U \\
&>\, \frac1K(K(U+L+R+\eps+2\Delta+2M(1,0)+6\kappa_1)-R-\eps)-U \\
&=\, L+2\Delta+2M(1,0)+6\kappa_1 \,>\, L.
\end{align*}

\begin{claim}
\label{claim:Delta_nbhd}
$\beta[0, t-U]\subset\mathcal N^G_{\Delta+M(1,0)}(\alpha)$.
\end{claim}

\begin{subproof}{Proof}
Consider a geodesic $[1,\pi_\alpha(v)]$ in $G$. Because $\alpha$ is $M$--Morse, $[1,\pi_\alpha(v)]$ lies in the $M(1,0)$--neighbourhood of $\alpha$, and is $M_+$--Morse, where $M_+=M+M(1,0)$. Now consider a geodesic $[u,\pi_\alpha(v)]$ in $G$. The geodesic triangle formed by $[1,\pi_\alpha(v)]$, $[\pi_\alpha(v),u]$, and $\beta[0,t]$ has two sides that are $M'$-- and $M_+$--Morse, so, according to \cite[Lem.~2.3]{cordes:morse}, the geodesic $[u,\pi_\alpha(v)]$ is uniformly Morse. Assumption~\ref{assumpt:space}.\eqref{item:morse_detectable} then tells us that $\rho[u,\pi_\alpha(v)]$ is a $\lambda'$--quasigeodesic, where $\lambda'=\lambda'(M',M_+)$. Since $u$ was constructed so that $\dist_X(ux_0,\pi_\alpha(v)x_0)\le\eps+R$, this implies that $\dist_G(u,\pi_\alpha(v))\le\lambda'(R+\eps)+\lambda'$. Ensure that $U$ is larger than this bound by at least $\Delta$.

From Lemma~\ref{lem:cordes_morse_triangles_thin}, we know that the above geodesic triangle is $\Delta$--thin. Thus, for $s<t-U$, the point $\beta(s)$ must be $\Delta$--close to $[1,\pi_\alpha(v)]$. In turn, this means that it is $(\Delta+M(1,0))$--close to $\alpha$. 
\end{subproof}

We write $t'=t-U-2\Delta-2M(1,0)-6\kappa_1$. Note that by the above computation, $t'>L$.

\begin{claim}
\label{claim:kappa_neigh}
$\beta[0,t'] \subseteq \mathcal N^G_{\kappa_2}(\alpha)$.
\end{claim}

\begin{subproof}{Proof}
If not, then by Lemma~\ref{lem:cashen_mackay}, there exists $T_0\le t'$ such that for all $s\ge0$ we have $\dist_G(\beta(s),\alpha)\ge\frac12(s-T_0)-2\kappa_1$. In particular, for $s=t-U$, Claim~\ref{claim:Delta_nbhd} leads to
\begin{align*}
\Delta+M(1,0) \,\ge\, \dist_G&(\beta(t-U,\alpha)) \,\ge\, \frac12(t-U-T_0)-2\kappa_1 \\
&\ge\, \frac12(t-U-t')-2\kappa_1 \,\ge\, \Delta+M(1,0)+\kappa_1,
\end{align*}
a contradiction.
\end{subproof}

Since $\alpha$ is a $\nu$--quasigeodesic and $\beta$ is a geodesic, it follows from Claim~\ref{claim:kappa_neigh} that there is an initial subsegment $\alpha'$ of $\alpha$ that stays $\kappa$--close to $\beta[0,t']$, where $\kappa$ is a uniform constant depending only on $\kappa_2$ and $\nu$. That is, $\kappa=\kappa(M,\nu)$. In other words, the Hausdorff-distance between $\alpha'$ and $\beta[0,t']$ is at most $\kappa$.

Since $t'>L$, the fact that $\beta$ is $(M,\kappa,L)$--incompatible means that there is some $(k,c)$--quasigeodesic $\mu$ with endpoints $\mu^-$ and $\mu^+$ on $\beta[0,t']$ such that there is some $p\in\mu$ with $\dist_G(p,\beta)>M(k,c+2\kappa)+2\kappa$. Let $y^-$ and $y^+$ be closest points in $\alpha$ to $\mu^-$ and $\mu^+$, respectively. We have $\dist(y^\pm,\mu^\pm)\le\kappa$. It follows that the path 
\[
[y^-,\mu^-]\cup\mu\cup[\mu^+,y^+]
\]
is a $(k,c+2\kappa)$--quasigeodesic with endpoints on the $M$--Morse quasigeodesic $\alpha'$. But this path contains $p\in\mu$, which is at a distance of more than $M(k,c+2\kappa)+2\kappa$ from $\beta[0,t']$. Since $\beta[0,t']$ and $\alpha'$ are at Hausdorff-distance at most $\kappa$, this contradicts the fact that $\alpha$ is $M$--Morse. We conclude from this contradiction that $\dist_G(1,\pi_\alpha(v))\le D_1$ for all $v\in\beta$.

We now turn to the other inequality. Given a point $v\in\alpha$, consider $w=\pi_\beta(v)$ and $y=\pi_\alpha(w)$. We aim to bound $\dist_G(1,w)$. From the above, we know that $\dist_G(1,y)\le D_1$.

Letting $R_\beta$ be given by applying Lemma~\ref{lem:exo_hyperbolicspace} to $\rho\beta$, we have $\dist_X(wx_0,\pi_{\rho\beta}(vx_0))\le R_\beta$ and $\dist_X(yx_0,\pi_{\rho\alpha}(wx_0))\le R$ by the comment after that lemma. Since $\rho\alpha$ and $\rho\beta$ are uniform quasigeodesics in the $\delta$--hyperbolic space $X$ issuing from the common point $x_0$, there is a uniform bound on $\dist_X(\pi_{\rho\beta}(vx_0),\pi_{\rho\alpha}\pi_{\rho\beta}(vx_0))$. Lemma~\ref{lem:exo_hyperbolicspace} also states that $\pi_{\rho\alpha}$ is $R$--coarsely Lipschitz, so by combining these bounds we obtain a uniform bound on $\dist_X(wx_0,yx_0)$.  

Because the restrictions of $\alpha$ and $\beta$ between 1 and $y$ and between 1 and $w$ are, respectively, $M$-- and $M'$--Morse, \cite[Lem.~2.3]{cordes:morse} tells us that a geodesic $[w,y]$ in $G$ is uniformly Morse. By Assumption~\ref{assumpt:space}.\eqref{item:morse_detectable}, it follows that $\rho[w,y]$ is a uniform quasigeodesic, and hence $\dist_G(w,y)$ is bounded by some uniform constant $\eps'$. 
We can now compute
\[
\dist_G(1,w) \,\le\, \dist_G(1,y)+\dist_G(y,w) \,\le\, D_1+\eps'. \qedhere
\]

\end{proof}

We now show the existence of an $(M, \kappa, L)$-incompatible ray, which combined with Lemma \ref{lem:small_gromov} will be enough to show that we can find two Morse rays that diverge after a short distance.

\begin{lemma}
\label{lem:existence_incompatible}
Let $G$ be a non-hyperbolic group acting on a hyperbolic space $X$ such that Assumption \ref{assumpt:space}.\eqref{item:morse_detectable} holds. For all Morse gauges $M$ and constants $\kappa$, there exist $M', L$ such that there is an $M'$-Morse ray which is $(M,\kappa, L)$-incompatible.
\end{lemma}

\begin{proof}
Given $M$ and $\kappa$, let $M_1(k,c)=M(k,c+2\kappa)+2\kappa$. Since $G$ is not hyperbolic, its Morse boundary is not compact \cite[Cor.~1.17]{cordesdurham:boundary}, so there must be some Morse ray $\beta$ that is not $M_1$-Morse. Let $M'$ be the Morse gauge of $\beta$. As $\beta$ is not $M_1$-Morse, there are some $k,c>0$ such that there is a $(k,c)$-quasi-geodesic $\mu$ with endpoints on $\beta$ that contains a point $p$ with $\dist_G(p,\beta)>M_1(k,c)$. Let $L$ be such that the endpoints of $\mu$ lie in $\beta[0,L]$. By definition, $\beta$ is $(M,\kappa, L)$--incompatible.
\end{proof}
 
\begin{corollary}
\label{cor:small_gromov_product_of_different_rays}
Let $G$ be a non-hyperbolic group acting on a hyperbolic space $X$ such that Assumption \ref{assumpt:space}.\eqref{item:morse_detectable} holds. For every Morse gauge $M$ there is a Morse gauge $M'$, a constant $D>0$, and an $M'$--Morse ray $\beta\subset G$, issuing from $1\in G$, such that if $\alpha$ is an $M$-Morse ray issuing from $1\in G$ then $\diam \left( \pi_{\alpha}(\beta) \right) \leq D$ and $\diam \left( \pi_{\beta}(\alpha) \right) \leq D$.
\end{corollary}

\begin{proof}
Let $\kappa=\kappa(M)$ be given by Lemma~\ref{lem:small_gromov}. By Lemma~\ref{lem:existence_incompatible}, there are $M'$ and $L$ such that there is an $M'$--Morse ray $\beta$ that is $(M,\kappa,L)$--incompatible. Now apply Lemma~\ref{lem:small_gromov}.
\end{proof}

\section{Bounded projections to axes}
\label{sec:bounded}

The goal of this section is to prove Proposition~\ref{prop:small_proj}, which is an important part of the proof of Theorem~\ref{thm:linear_progress} on linear progress in Section~\ref{sec:proof}. It asserts that there is a definite probability that tame, quasi-homogeneous Markov chains have bounded projection to any given loxodromic WPD.

\begin{proposition}
\label{prop:small_proj} 
Let $G$ be a non-hyperbolic group satisfying Assumption~\ref{assumpt:space} and let $(w_n^{*})$ be a tame, quasi-homogeneous Markov chain. Fix a loxodromic WPD element $g\in G$. There exist $C, \epsilon >0$ such that for all $p,h \in G$ and for all $n \in \mathbb N$, we have
	$$\mathbb P \Big[d_{hE(g)}(p,w_n^{p})\leq C\Big] > \epsilon.$$
\end{proposition}

Recall from Section~\ref{subsec:ggt_defns} that $E(g)$ is the elementary closure of $g$. See Notation~\ref{not:projections} for the definition of $d_{hE(g)}$. For ease of notation, we write $\gamma=E(g)$ for the remainder of this section. Let us write $\nu$ for the quasi-homogeneity constant of $(w_n^*)$. Recall that $\rho:G\to X$ is the orbit map with basepoint $x_0$.

\begin{lemma} 
\label{lem:homogeneous_morse} 
For each Morse gauge $M$ and $\nu \ge 1$, there is a constant $A>0$ such that the following holds for all $C>1$. Let $\xi\subset G$ be an $M$--Morse ray, and let $p'\in\xi$, $p\in G$. If $\phi:G\to G$ is a bijective $\nu$--quasi-isometry with $\phi(p')=p$, then 
$$ d_{\phi\xi}(p,\phi(h)) \,>\, \frac1Ad_{\xi}(p',h)-A$$
for all $h\in G$. Moreover,
\[
\mathbb P \Big[ d_{\phi\xi}(p,\phi_\#(w_n^{p'}))>C \Big] \,\ge\, \mathbb P \Big[ d_\xi(p',w_n^{p'}) >A(C+A)\Big].
\]
\end{lemma}

\begin{proof}
The first statement is essentially a consequence of Assumption~\ref{assumpt:space}.\eqref{item:exist_quasi_isometry}, that $\phi$ descends to a quasi-isometry of $X$, and Lemma~\ref{lem:quasi_isometries_and_projections}, which states that ``the translate (by a quasiisometry) of the projection of a point to a quasigeodesic coarsely agrees with the projection of the translate (of the point) to the translate (of the quasigeodesic)''. The nature of $X$--projections is such that a careful argument requires additional small errors. We now give details.

Let $\lambda\ge1$, given by Assumption~\ref{assumpt:space}, be such that $\xi x_0$ and $\phi(\xi)x_0$ are $\lambda$-quasi-geodesics. As noted after Lemma~\ref{lem:exo_hyperbolicspace}, if $Y\subset G$ has $\lambda$--quasiconvex orbit then there is a uniform constant $R$ such that the maps $\rho\pi_Y$ and $\pi_{Yx_0}\rho$ differ by at most $R$. Thus 
\[
\dist_{\phi\xi}(p,\phi(h)) \,\ge\, \dist_X(\pi_{\phi(\xi)x_0}(px_0),\pi_{\phi(\xi)x_0}(\phi(h)x_0))-2R.
\]
According to Assumption~\ref{assumpt:space}.\eqref{item:exist_quasi_isometry}, the maps $\rho\phi$ and $\bar\phi\rho$ differ by at most $\lambda$, so there is a uniform constant $R'$ such that 
\[
\dist_{\phi\xi}(p,\phi(h)) \,\ge\, \dist_X\big(\pi_{\bar\phi(\xi x_0)}(px_0),\pi_{\bar\phi(\xi x_0}(\bar\phi(hx_0))\big) - 2R'.
\]
Using Lemma~\ref{lem:quasi_isometries_and_projections}, there is now a uniform constant $N$ such that 
\[
\dist_{\phi\xi}(p,\phi(h)) \,\ge\, \dist_X\big(\bar\phi\pi_{\xi x_0}(\bar\phi^{-1}(px_0)),\bar\phi\pi_{\xi x_0}(hx_0)\big) - 2R'-2N.
\]
As $\rho\phi$ and $\bar\phi\rho$ coarsely agree, $\bar\phi^{-1}(px_0)$ is uniformly close to $p'x_0$, because $\phi(p')=p$. Also note that, by Lemma~\ref{lem:exo_hyperbolicspace}, the map $\pi_{\xi x_0}$ is uniformly coarsely Lipschitz. Hence the fact that $\bar\phi$ is a $\lambda$--quasi-isometry means that there is a uniform constant $R''$ such that 
\[
\dist_{\phi\xi}(p,\phi(h)) \,\ge\, \frac1\lambda \dist_X(\pi_{\xi x_0}(p'x_0),\pi_{\xi x_0}(hx_0))-R''.
\]
To complete the proof of the first statement, observe that since $\rho\pi_{\xi}$ and $\pi_{\xi x_0}\rho$ differ by at most $R$, we have 
\[
\dist_X(\pi_{\xi x_0}(p'x_0),\pi_{\xi x_0}(hx_0)) \,\ge\, \dist_\xi(p',h)-2R.
\]

For the second statement, let $C>1$. By the first statement, any $h\in G$ with $\dist_\xi(p',h)>A(C+A)$ also satisfies $\dist_{\phi\xi}(p,\phi(h))>C$. We can therefore use quasi-homogeneity and the fact that $\phi$ is a bijection to compute
\begin{align*}
\mathbb P\big[\dist_\xi(p',w^{p'}_n)>A(C+A)\big] \,&=\,
    \sum_{\dist_\xi(p',h)>A(C+A)}\mathbb P\big[w^{p'}_n=h\big] \\
&\le \sum_{\dist_{\phi\xi}(p,\phi(h))>C}\mathbb P\big[w^{p'}_n=h\big] \\
&=\sum_{\dist_{\phi\xi}(p,h')>C}\mathbb P\big[w^{p'}_n=\phi^{-1}(h')\big] \\
&=\, \mathbb P\big[\dist_{\phi\xi}(p,\phi_\#(w^{p'}_n))>C\big]. \qedhere
\end{align*}
\end{proof}
	 	
The following lemma says that a tame Markov chain has a (small but) positive probability of reaching a given point at distance $d$ within a number of steps linear in $d$.

\begin{lemma}[{\cite[Lemma 2.9]{goldsboroughsisto:markov}}]
\label{lem:get_anywhere}
If $(w^*_n)$ is a tame Markov chain on a finitely generated group $G$, then there are constants $U, \epsilon_0 >0$ such that the following holds. For each $p,q \in G$ with $d_G(p,q)=d$ there exists $t\leq dU$ such that 
$$ \mathbb P [ w^{q}_t=p]\geq \epsilon^{d}_0.$$
\end{lemma}

As a final preliminary lemma, the following will allow us to ``pivot away'' certain quasi-geodesics in $X$ from a given translate $h\rho\gamma$ of $\rho\gamma=E(g)x_0$. The lemma is inspired by \cite[Lemma 9.5]{MathieuSisto} (but the proof is different and simpler).

\begin{lemma}\label{lem:pivot}
For all $\theta$ there exist $s,E$ such that the following holds. Given a $\theta$-quasi-geodesic $\alpha\subset X$ from $px_0$ that passes through $qx_0$, and given $h\in G$, there exists $q'\in G$ with $d_G(p,q')\leq s$ such that 
$$d_{h\rho\gamma}(q'q^{-1}\alpha,px_0)\leq E\ \ \ \textrm{and}\ \ \ d_{q'q^{-1}\alpha}(h\rho\gamma,q'x_0)\leq E.$$
\end{lemma}

\begin{proof}
To simplify notation, let us write $\alpha'=pq^{-1}\alpha$. The desired element $q'\in G$ will be of the form $q'=pf$, where $\dist_G(1,f)\le s$. We shall then have $q'q^{-1}\alpha=pfp^{-1}\alpha'$.

Fix any two loxodromic WPD elements $g_1,g_2\in G$ such that $E(g_1)$, $E(g_2)$, and $E(g)=\gamma$ are pairwise distinct, the existence of which is given by the arguments in \cite[Proposition~6]{bestvinafujiwara:bounded} (see also \cite[Corollary~6.12]{dahmaniguirardelosin:hyperbolically}). Observe that all three of these contain $1\in G$. We shall choose $f=g_i^j$ for some $i$ and some bounded $j$. 

At most one direction of $E(g_i)x_0$ can have large coarse intersection with the initial subsegment of $p^{-1}\alpha'$ from $q^{-1}px_0$ to $x_0$, so after inverting the $g_i$, we may assume that the positive direction has small coarse intersection with it. Write $\gamma_i$ for this positive direction of $E(g_i)$. There is some constant $C_0$ such that for any $k,k'\in G$, both $\pi_{k\rho\gamma}(k'\rho\gamma_i)$ and $\pi_{k'\rho\gamma_i}(k\rho\gamma)$ have diameter at most $C_0$; see \cite[Thm~3.9]{AMS:commensurating}, for instance (said reference gives geometric separation, which is equivalent to bounded projections). The conclusion of Lemma~\ref{lem:behrtock_inequality} (the Behrstock inequality) therefore holds, with some constant $B$, for all distinct pairs of translates of $\rho\gamma$, $\rho\gamma_1$, and $\rho\gamma_2$. 

We now show that there exists $C_1$ such that a least one of the $\pi_{p\rho\gamma_i}(\alpha')$ has diameter at most $C_1$. If this were not the case, then there would exist $x_1,x_2\in\alpha'$ such that $\dist_{p\rho\gamma_i}(px_0,x_i)$ were large. By Lemma~\ref{lem:exo_hyperbolicspace}, we then also have that $\dist_X(px_0,x_i)$ is large. Because the $p\rho\gamma_i$ have small coarse intersection with the initial subsegment of $\alpha'$ between $pq^{-1}px_0$ and $px_0$, we can relabel so that $x_1$ lies between $px_0$ and $x_2$. Because the conclusion of Lemma~\ref{lem:behrtock_inequality} holds for $p\rho\gamma_1$ and $p\rho\gamma_2$, the set $\gamma'_p=p\rho\gamma_1\cup p\rho\gamma_2$ is quasiconvex. Lemma~\ref{lem:exo_hyperbolicspace} therefore tells us that any geodesic from $x_1$ to $x_2$ must pass uniformly close to $px_0$. By the Morse lemma, this contradicts the fact that $\alpha'$ is a quasigeodesic.

After relabelling, we therefore have that $\diam_X(\pi_{p\rho\gamma_1}(\alpha'))\le C_1$. Because $pg_1p^{-1}$ acts on $pE(g_1)x_0$ with positive translation length independently of $\alpha$, there is some uniformly bounded $j$ such that the set $\pi_{p\rho\gamma_1}(pg_1^jp^{-1}\alpha')=pg_1^jp^{-1}\pi_{p\rho\gamma_1}(\alpha')$ lies at distance greater than $B$ from $\pi_{p\rho\gamma_1}(h\rho\gamma)$. Setting $q'=pg_1^j$ (i.e. $f=g_1^j$), the desired inequalities follow from the Behrstock inequality, Lemma~\ref{lem:behrtock_inequality}.

\end{proof}

We are now ready to prove the main proposition. Before doing so, we fix various constants and Morse gauges. The dependencies are rather delicate, so we have to be quite careful here. The reader may prefer to skip this and refer back when checking that the various constants and Morse gauges have the claimed properties.

\begin{constantsgauges} \quad

\begin{itemize}
\item   $M_0$ is such that $G$ contains an $M_0$--Morse ray (Assumption~\ref{assumpt:space}.\eqref{item:exist_wpd}).
\item   $\nu$ is the quasi-homogeneity constant (Assumption~\ref{assumpt:markov_chain_quasihomogeneous}).
\item   $M$ is the minimal Morse gauge such that whenever $\phi,\psi$ are bijective $\nu$--quasi-isometries of $G$, any $M_0$--Morse ray gets mapped by $\phi^{-1}\psi$ to an $M$--Morse ray.
\item   $M'$ and $D$ are given by applying Corollary \ref{cor:small_gromov_product_of_different_rays} to the gauge $M$.
\item   $A$ is given by applying Lemma~\ref{lem:homogeneous_morse} to the gauge $\max\{M_0,M,M'\}$.
\item   $U$ and $\epsilon_0$ are as in Lemma \ref{lem:get_anywhere}.
\item   $\theta$ is such that the $X$--orbit of any $\max\{M_0,M,M'\}$--Morse geodesic in $G$ is a $\theta$--quasiconvex $\theta$ quasi-geodesic (Assumption~\ref{assumpt:space}.\eqref{item:morse_detectable}). 
\item   $s$ and $E$ are given by applying Lemma~\ref{lem:pivot} to $\theta$.
\item   $R$ and $B$ are given by applying Lemmas~\ref{lem:exo_hyperbolicspace} and~\ref{lem:behrtock_inequality}, respectively, with quasi-convexity constant $\theta$.
\item   $C_1=A(B+A+D+1)$ and $C_2=A(C_1+A)$.
\item   $d=C_2+\theta(B+E+2R+\theta)$.
\item   $J>0$ is a constant such that for any $h,p\in G$, consecutive points in $(\pi_{h\gamma}w^p_n)$ have distance at most $J$. This exists by tameness of the Markov chain and Lemma \ref{lem:exo_hyperbolicspace}.
\item   $C=d+JU(d+s)$.
\end{itemize}
\end{constantsgauges}

\begin{proof}[Proof of Proposition \ref{prop:small_proj}]
For a contradiction, assume that Proposition~\ref{prop:small_proj} does not hold. Hence there exist $p,h \in G$ and $n \in \mathbb N$ such that: \begin{equation}
\label{eqn:contradiction_assumed}
\mathbb P \Big[d_{h\gamma}(p,w_n^{p})\leq C\Big] \,\le\, \frac13\epsilon_0^{d+s},
\end{equation}
where $C$ is as above. Let $\beta$ be an $M'$-Morse ray issuing from $p$ that satisfies the conclusion of Corollary~\ref{cor:small_gromov_product_of_different_rays}, and let $\alpha$ be an $M_0$-Morse ray issuing from $p$. 

\begin{figure}[ht]
  \tikzset{every picture/.style={line width=0.75pt}} 

\begin{tikzpicture}[x=0.75pt,y=0.75pt,yscale=-1,xscale=1]

\draw [color={rgb, 255:red, 208; green, 2; blue, 27 }  ,draw opacity=1 ]   (125,280) -- (633,255) ;
\draw  [line width=1.5] [line join = round][line cap = round] (408,51) .. controls (408,48.67) and (408,46.33) .. (408,44) ;
\draw  [line width=1.5] [line join = round][line cap = round] (417.45,54.1) .. controls (415.15,53.7) and (412.85,53.3) .. (410.55,52.9) ;
\draw  [line width=1.5] [line join = round][line cap = round] (419,51) .. controls (419,48.67) and (419,46.33) .. (419,44) ;
\draw  [line width=1.5] [line join = round][line cap = round] (425.96,36.63) .. controls (423.99,37.88) and (422.01,39.12) .. (420.04,40.37) ;
\draw  [line width=1.5] [line join = round][line cap = round] (435.77,40.64) .. controls (433.92,39.21) and (432.08,37.79) .. (430.23,36.36) ;
\draw  [line width=1.5] [line join = round][line cap = round] (443.71,46.71) .. controls (441.9,45.24) and (440.1,43.76) .. (438.29,42.29) ;
\draw  [line width=1.5] [line join = round][line cap = round] (441.08,55.42) .. controls (442.36,53.47) and (443.64,51.53) .. (444.92,49.58) ;
\draw  [line width=1.5] [line join = round][line cap = round] (441,66) .. controls (441,63.67) and (441,61.33) .. (441,59) ;
\draw  [line width=1.5] [line join = round][line cap = round] (408,129) .. controls (408,126.67) and (408,124.33) .. (408,122) ;
\draw  [line width=1.5] [line join = round][line cap = round] (407,119) .. controls (407,116.67) and (407,114.33) .. (407,112) ;
\draw  [line width=1.5] [line join = round][line cap = round] (400.4,104.15) .. controls (402.13,105.72) and (403.87,107.28) .. (405.6,108.85) ;
\draw  [line width=1.5] [line join = round][line cap = round] (390.5,101.47) .. controls (392.83,101.49) and (395.17,101.51) .. (397.5,101.53) ;
\draw  [line width=1.5] [line join = round][line cap = round] (379.5,101.64) .. controls (381.83,101.55) and (384.17,101.45) .. (386.5,101.36) ;
\draw  [line width=1.5] [line join = round][line cap = round] (373.14,110.89) .. controls (373.71,108.63) and (374.29,106.37) .. (374.86,104.11) ;
\draw  [line width=1.5] [line join = round][line cap = round] (379.48,119.97) .. controls (377.83,118.32) and (376.17,116.68) .. (374.52,115.03) ;
\draw  [line width=1.5] [line join = round][line cap = round] (387.93,125.41) .. controls (385.98,124.14) and (384.02,122.86) .. (382.07,121.59) ;
\draw  [line width=1.5] [line join = round][line cap = round] (394.9,134.44) .. controls (393.63,132.48) and (392.37,130.52) .. (391.1,128.56) ;
\draw  [line width=1.5] [line join = round][line cap = round] (392.58,144.7) .. controls (393.53,142.57) and (394.47,140.43) .. (395.42,138.3) ;
\draw  [line width=1.5] [line join = round][line cap = round] (394.72,155.93) .. controls (394.24,153.64) and (393.76,151.36) .. (393.28,149.07) ;
\draw  [line width=1.5] [line join = round][line cap = round] (393.75,165.99) .. controls (393.92,163.66) and (394.08,161.34) .. (394.25,159.01) ;
\draw  [line width=1.5] [line join = round][line cap = round] (391.34,176.94) .. controls (391.78,174.65) and (392.22,172.35) .. (392.66,170.06) ;
\draw  [line width=1.5] [line join = round][line cap = round] (392.13,186.18) .. controls (392.11,184.86) and (392.1,183.54) .. (392.08,182.22) ;
\draw  [line width=1.5] [line join = round][line cap = round] (436.12,75.45) .. controls (437.37,73.48) and (438.63,71.52) .. (439.88,69.55) ;
\draw  [line width=1.5] [line join = round][line cap = round] (436,86) .. controls (436,83.67) and (436,81.33) .. (436,79) ;
\draw  [line width=1.5] [line join = round][line cap = round] (436,96) .. controls (436,93.67) and (436,91.33) .. (436,89) ;
\draw  [line width=1.5] [line join = round][line cap = round] (427.81,100.94) .. controls (429.94,99.98) and (432.06,99.02) .. (434.19,98.06) ;
\draw  [line width=1.5] [line join = round][line cap = round] (419.95,108.34) .. controls (421.32,106.45) and (422.68,104.55) .. (424.05,102.66) ;
\draw  [line width=1.5] [line join = round][line cap = round] (419,118) .. controls (419,115.67) and (419,113.33) .. (419,111) ;
\draw  [line width=1.5] [line join = round][line cap = round] (419,129) .. controls (419,126.67) and (419,124.33) .. (419,122) ;
\draw  [line width=1.5] [line join = round][line cap = round] (409.5,131.41) .. controls (411.83,131.47) and (414.17,131.53) .. (416.5,131.59) ;
\draw  [line width=1.5] [line join = round][line cap = round] (418.2,9.91) .. controls (418.73,7.64) and (419.27,5.36) .. (419.8,3.09) ;
\draw  [line width=1.5] [line join = round][line cap = round] (418,21) .. controls (418,18.67) and (418,16.33) .. (418,14) ;
\draw  [line width=1.5] [line join = round][line cap = round] (412.27,30.54) .. controls (413.42,28.51) and (414.58,26.49) .. (415.73,24.46) ;
\draw  [line width=1.5] [line join = round][line cap = round] (409.15,39.89) .. controls (409.72,37.63) and (410.28,35.37) .. (410.85,33.11) ;
\draw [color={rgb, 255:red, 74; green, 144; blue, 226 }  ,draw opacity=1 ]   (427,6) -- (517.62,101.55) ;
\draw [shift={(519,103)}, rotate = 226.52] [color={rgb, 255:red, 74; green, 144; blue, 226 }  ,draw opacity=1 ][line width=0.75]    (10.93,-3.29) .. controls (6.95,-1.4) and (3.31,-0.3) .. (0,0) .. controls (3.31,0.3) and (6.95,1.4) .. (10.93,3.29)   ;
\draw [color={rgb, 255:red, 126; green, 211; blue, 33 }  ,draw opacity=1 ]   (414,5) -- (239.88,69.31) ;
\draw [shift={(238,70)}, rotate = 339.73] [color={rgb, 255:red, 126; green, 211; blue, 33 }  ,draw opacity=1 ][line width=0.75]    (10.93,-3.29) .. controls (6.95,-1.4) and (3.31,-0.3) .. (0,0) .. controls (3.31,0.3) and (6.95,1.4) .. (10.93,3.29)   ;
\draw [color={rgb, 255:red, 208; green, 2; blue, 27 }  ,draw opacity=1 ]   (394,197) -- (396.91,265) ;
\draw [shift={(397,267)}, rotate = 267.55] [color={rgb, 255:red, 208; green, 2; blue, 27 }  ,draw opacity=1 ][line width=0.75]    (10.93,-3.29) .. controls (6.95,-1.4) and (3.31,-0.3) .. (0,0) .. controls (3.31,0.3) and (6.95,1.4) .. (10.93,3.29)   ;
\draw [color={rgb, 255:red, 208; green, 2; blue, 27 }  ,draw opacity=1 ]   (414,5) -- (219.17,273.38) ;
\draw [shift={(218,275)}, rotate = 305.98] [color={rgb, 255:red, 208; green, 2; blue, 27 }  ,draw opacity=1 ][line width=0.75]    (10.93,-3.29) .. controls (6.95,-1.4) and (3.31,-0.3) .. (0,0) .. controls (3.31,0.3) and (6.95,1.4) .. (10.93,3.29)   ;
\draw [color={rgb, 255:red, 126; green, 211; blue, 33 }  ,draw opacity=1 ]   (167,34) .. controls (207,4) and (354,220) .. (394,190) ;
\draw [color={rgb, 255:red, 74; green, 144; blue, 226 }  ,draw opacity=1 ]   (394,190) .. controls (434,160) and (560,107) .. (565,13) ;


\draw (122,130) node [anchor=north west][inner sep=0.75pt]   [align=left] {};
\draw (591.79,240.5) node [anchor=north west][inner sep=0.75pt]  [opacity=1 ,rotate=-358.53] [align=left] {$hE(g)$};
\draw (122,143) node [anchor=north west][inner sep=0.75pt]   [align=left] {};


\filldraw[black] (394,190) circle (2pt) node[anchor=west]{$p$};

\filldraw[draw opacity=1] (219.17,273.38) circle (1pt) node[anchor=north]{$\pi_{hE(g)}(w_n^p)$};

\filldraw[black] (419,5.5) circle (2pt) node[anchor=south]{$w_n^p$};

\filldraw[draw opacity=1 ](239.88,69.31) circle (1pt) node[anchor=east]{$\pi_{\beta}(w_n^p)$};

\filldraw[draw opacity=1](396.91,265) circle (1pt) node[anchor=north]{$\pi_{hE(g)}(p)$};

\filldraw[draw opacity=1 ](517.62,101.55) circle (1pt) node[anchor=west]{$\pi_{\alpha}(w_n^p)$};

\draw[dashed, color={rgb, 255:red, 208; green, 2; blue, 27 }  ,draw opacity=1, <->] (396.91,265)--(219.17,273.38);

\draw[draw opacity=1](320,255) node{$\geq C$};

\end{tikzpicture}
  \caption{For a contradiction, we assume that with overwhelming probability the Markov chain creates a large projection on $hE(g)$. The dotted line is meant to represent the sample path of the Markov chain and $\alpha$, $\beta$ are the two Morse rays with very different Morse gauges.}
  \label{fig:tikz:my}
\end{figure}

The main ingredient in producing a contradiction is the following claim, which states that if \eqref{eqn:contradiction_assumed} holds then in a definite number of steps, the Markov chain $(w^p_n)$ simultaneously travels a long way along two Morse rays that are in very different directions. 

\begin{claim*}
\label{claim:two_diff_directions}
There exists $m \in \mathbb N$ such that there are bijective $\nu$--quasi-isometries $\phi_\alpha,\phi_\beta:G\to G$ with $\phi_\alpha(p)=\phi_\beta(p)=p$ such that 
\[
\mathbb P\big[ d_{\phi_\alpha\alpha}(p, w_m^{p})>C_1\big]\ge\frac23 \quad \text{and} \quad \mathbb P\big[ d_{\phi_\beta\beta}(p, w_m^p) > C_1 \big] \ge\frac23.
\]
\end{claim*}

\begin{subproof}{Proof}
We first fix the value of $m$ to be considered. Let $q\in\beta$ have $\dist_G(p,q)=d$, where $d$ is as above. According to Lemma~\ref{lem:pivot}, there is some $q'\in G$, with $\dist_G(p,q')\le s$, such that $\diam_X(\pi_{h\rho\gamma}(q'q^{-1}\rho\beta))\le E$. Let $p'=q'q^{-1}p$, so that $\dist_G(p',q')=\dist(p,q)=d$.

Because $\dist_G(p,p')\le d+s$, Lemma~\ref{lem:get_anywhere} tells us that there is some $t_\beta\le(d+s)U$ such that $\mathbb P[w^p_{t_\beta}=p']\ge\epsilon_0^{d+s}$. An analogous construction yields a corresponding number $t_\alpha\le(d+s)U$. Observe that, by \eqref{eqn:contradiction_assumed} and the choice of $J$, 
\[
\mathbb P\Big[\dist_{h\gamma}(p,w^p_{n+t_\beta-t_\alpha}) \le C-JU(d+s)=d\Big] \,\le\, \frac13\epsilon_0^{d+s}.
\]
Set $m=n-t_\alpha$. Let us write $\hat n=n+t_\beta-t_\alpha$ and $t=t_\beta$. We have $m=\hat n-t$.

The above probability that $\dist_{h\gamma}(p,w^p_{\hat n})$ is at most $d$ is bounded below by the probability of the Markov chain going to $p'$ in exactly $t$ steps and then making little distance along $h\gamma$ for a further $\hat n-t$ steps. This gives us
\[
\epsilon 
\,\ge\, \mathbb P\Big[w^p_t=p'\Big] \cdot \mathbb P\Big[\dist_{h\gamma}(p,w^{p'}_{\hat n-t}) \le d\Big]
\,\ge\, \epsilon_0^{d+s}\mathbb P\Big[\dist_{h\gamma}(p,w^{p'}_{\hat n-t}) \le d\Big],
\]
from which we deduce that $\mathbb P[\dist_{h\gamma}(p,w^{p'}_{\hat n-t})\le d]\le\frac13$. Since $m=\hat n-t$, we have established that 
\[
\mathbb P\Big[\dist_{h\gamma}(p,w^{p'}_m)>d\Big]\,\ge\,\frac23.
\]

Now let us consider an arbitrary point $k\in G$ satisfying $\dist_{h\gamma}(p,k)>d$. Since $h\rho\gamma$ is a $\theta$--quasigeodesic, Lemma~\ref{lem:exo_hyperbolicspace} tells us that $\dist_{h\rho\gamma}(px_0,kx_0)>\frac{d}\theta-\theta-R$. By the choice of $q'$ via Lemma~\ref{lem:pivot}, we therefore have 
\[
\dist_{h\rho\gamma}(kx_0,q'q^{-1}\rho\beta) \,>\, \frac{d}\theta-\theta-R-E \,>\, B.
\]
By the Behrstock inequality, Lemma~\ref{lem:behrtock_inequality}, we have $\dist_{q'q^{-1}\rho\beta}(h\rho\gamma,kx_0)\le B$. Now, using Lemmas~\ref{lem:pivot} and~\ref{lem:exo_hyperbolicspace}, the fact that $q'q^{-1}\rho\beta$ is a $\theta$--quasigeodesic means that this leads to
\begin{align*}
\dist_{q'q^{-1}\beta}(p',k) \,&\ge\, \dist_{q'q^{-1}\beta}(p',q')-\dist_{q'q^{-1}\beta}(q',k) \\
&\ge\, d - \theta(\dist_{q'q^{-1}\rho\beta}(q'x_0,kx_0)+2R+\theta) \\
&\ge\, d-\theta(B+E+2R+\theta) \,=\, C_2.
\end{align*}
In particular, our previous estimate on the Markov chain implies that
\[
\mathbb P\Big[\dist_{q'q^{-1}\beta}(p',w^{p'}_m)>C_2\Big] \,\ge\,\frac23.
\]

By quasi-homogeneity, Assumption~\ref{assumpt:markov_chain_quasihomogeneous}, there is a bijective $\nu$--quasiisometry $\phi:G\to G$ that respects the Markov chain and has $\phi(p')=p$. Applying Lemma~\ref{lem:homogeneous_morse} with $\xi=q'q^{-1}\beta$, we therefore have
\begin{align*}
\mathbb P\Big[\dist_{\phi q'q^{-1}\beta}(p,w^p_m)>C_1\Big] \,
    &=\, \mathbb P\Big[\dist_{\phi q'q^{-1}\beta}(p,\phi_\#(w^{p'}_m))>C_1\Big] \\
&\ge\, \mathbb P\Big[\dist_{q'q^{-1}\beta}(p',w^{p'}_m)>A(C_1+A)\Big] \\
&\ge\, \mathbb P\Big[\dist_{q'q^{-1}\beta}(p',w^{p'}_m)>C_2\Big] \,\ge\, \frac23.
\end{align*}
The statement follows by letting $\phi_\beta$ be the quasiisometry $\phi q'q^{-1}$ of $X$, which depends on $\beta$ because $q$ and $q'$ do.

\end{subproof}

In view of the claim, there must be some $q\in G$ such that $\dist_{\phi\alpha}(p,q)>C_1$ and $\dist_{\phi\beta}(p,q)>C_1$. Let $\alpha'=\phi_\beta^{-1}\phi_\alpha\alpha$ and $q'=\phi_\beta^{-1}(q)$. Because $\phi_\beta(p)=p$, the first part of Lemma~\ref{lem:homogeneous_morse} tells us that 
\[
\dist_{\alpha'}(p,q') \,>\, \frac1A\dist_{\phi_\alpha\alpha}(p,q)-A \,>\, \frac{C_1}A-A.
\]
The choice of $\beta$ via Corollary~\ref{cor:small_gromov_product_of_different_rays} means that $\diam(\pi_\alpha(\beta))\le D$, so since $p\in\beta$, we have $\dist_{\alpha'}(q',\beta)>B$ by definition of $C_1$. Similarly, $\dist_\beta(q',\alpha)>B$. This contradicts the Behrstock inequality, Lemma~\ref{lem:behrtock_inequality}.

\end{proof}

\section{Linear progress} \label{sec:proof}

Our main theorem on Markov chains is the following.

\begin{theorem}
\label{thm:linear_progress}
Let $G$ be a group acting on a hyperbolic space $X$ with basepoint $x_0$ and satisfying Assumption \ref{assumpt:space}. If $(w_n^{*})_n$ is a tame, quasi-homogeneous Markov chain on $G$, then there exists a constant $C$ such that for all $o\in G$ and $n \in \mathbb N$ we have:
$$ \mathbb P \Big[ d_X(ox_0, w_n^{o}x_0) \geq n/C\big] \,\geq\, 1-Ce^{-n/C}.$$
\end{theorem}

Recall that a Markov chain satisfying the conclusion of the above theorem is said to make \textit{linear progress with exponential decay} in the hyperbolic space $X$.

In this section, we prove Theorem \ref{thm:linear_progress}. The main technical result that we establish in order to do this is Proposition~\ref{prop:axiom} below, the conclusion of which is the same as that of \cite[Proposition~5.1]{goldsboroughsisto:markov}. Theorem~\ref{thm:linear_progress} is a consequence of this by the following.

\begin{theorem}[{\cite[\S6]{goldsboroughsisto:markov}}]
\label{thm:GS_sec6}
Let $G$ be a group acting on a hyperbolic space $X$. Every tame Markov chain on $G$ satisfying the conclusion of \cite[Prop.~5.1]{goldsboroughsisto:markov} makes linear progress in $X$ with exponential decay.
\end{theorem}

Thus, in order to prove Theorem~\ref{thm:linear_progress} it suffices to establish Proposition~\ref{prop:axiom}, because Proposition~\ref{prop:small_proj} shows that, under the assumptions of Theorem~\ref{thm:linear_progress}, the hypotheses of Proposition~\ref{prop:axiom} are met. Before we can state Proposition~\ref{prop:axiom}, we need to introduce some notation.

\subsection{Notation and preliminary lemmas}\label{subsec:linear-progress-prelim-lemmas}
Fix a group $G$ acting on a hyperbolic space $X$, and a basepoint $x_0\in X$. We extend the notation used in Notation \ref{not:projections} with the following lemma, which states that, in the case of an axis of a WPD element, one can perturb closest-point projections to make them satisfy a stronger version of the Behrstock inequality.

\begin{lemma}[{\cite[Thm~4.1]{bestvinabrombergfujiwarasisto:acylindrical}}]
\label{lem:Behrstock_for_gamma}
Let $g\in G$ be a loxodromic WPD. Writing $\gamma=E(g)$, there is a constant $B$ and a $g$-equivariant map $\pi_{\gamma}:G\to \mathcal P(\gamma)$ with the following property, where $\mathcal P(\gamma)$ is the set of all subsets of $\gamma$. For all $x\in G$ and distinct translates $h\gamma \neq h'\gamma$, 
\[ 
\text{if }\, d_X(\pi_{h\gamma}(x),\pi_{h\gamma}(h'\gamma)) >B,
\,\text{ then }\, \pi_{h'\gamma}(x)=\pi_{h'\gamma}(h\gamma),
\]
where we define $\pi_{k\gamma}(z)=k\pi_{\gamma}(k^{-1}z)$. Moreover, for all $x\in G$ the Hausdorff distance between $\pi_\gamma(x)$ and any $X$-projection of $x$ to $\langle g\rangle$ is bounded by $B$.
\end{lemma}

In view of this lemma and with an abuse of notation, for $h,x,y\in G$ we will denote
$$d_{h\gamma}(x,y)=\diam (\pi_{h\gamma}(x)\cup\pi_{h\gamma}(y)),$$ 
where $\pi_{h\gamma}$ is the equivariant map of Lemma~\ref{lem:Behrstock_for_gamma}. 

The following definition captures the set of cosets where two given elements have far away projections.


\begin{definition}[{\cite[Def. 3.11]{goldsboroughsisto:markov}}]
\label{defn:H_T}
Let $g$ be a loxodromic WPD element of a group $G$ acting on a hyperbolic space. Given $x,y\in G$ and $T\ge0$, we write
\[
\mathcal H_T(x,y) \,=\, \{hE(g)\,:\,\dist_{hE(g)}(x,y)\ge T\}.
\]
With $g$ fixed, let $o,p,x,y \in G$. We define the following ``distance formula'' expression:
\[
\sum_{\mathcal{H}_T(o,p)}[x,y] \,\coloneq\, \sum_{\substack{hE(g) \,\in\, \mathcal{H}_T(o,p) \\ \pi_{hE(g)}(x) \,\neq\, \pi_{hE(g)}(y)}} d_{hE(g)}(x,y).
\]
\end{definition}

\begin{remark}[{\cite[Rem.~3.12]{goldsboroughsisto:markov}}]
 \label{rem:triangle_inequality}
Since projection distances satisfy the triangle inequality, so do these distance-formula expressions. That is, for all $o,p,x,y,z, \in G$ we have
\[
\sum_{\mathcal{H}_T(o,p)} [x,z] \,\leq \sum_{\mathcal{H}_T(o,p)} [x,y]  \,+ \sum_{\mathcal{H}_T(o,p)} [y,z].
\] 
\end{remark}

We can now state the main result of this section, the conclusion of which is the same as that of \cite[Prop.~5.1]{goldsboroughsisto:markov}. As discussed, Theorem~\ref{thm:linear_progress} follows from it, Theorem~\ref{thm:GS_sec6}, and Proposition~\ref{prop:small_proj}.

\begin{proposition}
\label{prop:axiom}
Let $G$ be a group acting on a hyperbolic space $X$ and fix a basepoint $x_0\in X$. Suppose that the conclusion of Proposition \ref{prop:small_proj} holds for some loxodromic WPD $g\in G$. There exist $T_0,C'$ such that the following holds for each $T\geq T_0$. For all $o,p\in G$, $n \in \mathbb{N}$, and $t>0$, 
$$
\mathbb{P}\Big[\exists r\leq n \,:\, \sum_{\mathcal{H}_T(o,p)} [p,w^p_r] \geq t \Big] \leq 2e^{-t/C'}.
$$
\end{proposition}

Before proceeding to the proof, we collect two more preliminary lemmas that are needed in order to choose the constant $T_0$.

\begin{lemma}[{\cite[Thm~4.1]{bestvinabrombergfujiwarasisto:acylindrical}, \cite[Thm~3.3(G)]{bestvinabrombergfujiwarasisto:acylindrical}}]
\label{lem:linearorder}
Fix a loxodromic WPD element $g\in G$, and let $o,p\in G$. Write $\gamma=E(g)$. For any sufficiently large $T$, the set $\mathcal{H}_T(o,p) \cup \{o\gamma,p\gamma\}$ is totally ordered with least element $o\gamma$ and greatest element $p\gamma$. The order is given by $h\gamma \prec h'\gamma$ if any one of the following equivalent conditions holds, where $B$ is as in Lemma~\ref{lem:Behrstock_for_gamma}.
\begin{itemize}
    \item $d_{h\gamma}(o,h'\gamma) > B$.
    \item $\pi_{h'\gamma}(o)= \pi_{h'\gamma}(h\gamma)$.
    \item $d_{h'\gamma}(p,h\gamma) > B$.
    \item $\pi_{h\gamma}(p)= \pi_{h\gamma}(h'\gamma)$.
\end{itemize}
\end{lemma}

\begin{lemma}[{\cite[Lem.~3.14]{goldsboroughsisto:markov}}]
\label{lem:GS_DF_lower_bound}
Let $g\in G$ be a loxodromic WPD and let $x_0\in X$. For all sufficiently large $T$, we have the following for all $a,b\in G$:
\[
\dist_X(ax_0,bx_0) \,\ge\, \frac12\sum_{\mathcal H_T(a,b)}[a,b].
\]
\end{lemma}

\subsection{Proof of Proposition~\ref{prop:axiom}}

Fix an element $g\in G$ satisfying the conclusion of Proposition~\ref{prop:small_proj}. Throughout the proof, we write $\gamma=E(g)$.
\begin{constants} \quad 
\begin{itemize} 
\item   $B$ is given by Lemma~\ref{lem:Behrstock_for_gamma}.
\item   $T'\ge10B$ is sufficiently large that the conclusions of Lemmas~\ref{lem:linearorder} and~\ref{lem:GS_DF_lower_bound} hold.
\item   $L=L(T')$ is given by \cite[Lem.~3.16]{goldsboroughsisto:markov}. 
\item   $J$ is such that for any $p,h\in G$, consecutive points of both $(w^p_nx_0)$ and $(\pi_{h\gamma}w^p_n)$ have distance at most $J$. This exists by tameness, Definition \ref{defn:tame}.
\item   $C$ and $\epsilon$ are as in Proposition \ref{prop:small_proj}. We can and will assume that $C>B$.
\item   $T_0=\max\{T',2(C+B+J+1)\}$.
\item   $D=T_0+LJ+L+C$.
\end{itemize}
\end{constants}

Fix $o,p, q\in G$ and $n \in \mathbb N$. For ease of notation, given $t>0$ we define 
\[
g(t)= \mathbb{P}\Big[\exists r\leq n : \sum_{\mathcal{H}_T(o,p)} [p,w^p_r] \geq t \Big], \quad\text{ and }\quad 
f(t)= \mathbb P\Big[\sum_{\mathcal{H}_T(o,p)} [p,w^p_n] \geq t \Big].
\]
Thus our goal is to find a constant $C_2$, independent of $t$, such that $g(t)\le C_2e^{-t/C_2}$. Fix $T\ge T_0$.

The main technical step in our proof of Proposition~\ref{prop:axiom} is the following, which relates the probability functions $f(t)$ and $g(t)$.

\begin{lemma}
\label{lem:bound_on_g(t)}
For all $t\geq D$, we have $\epsilon g(t) \leq f(t-D)-f(t+D)$. 
\end{lemma}

\begin{proof}
For ease of notation, set $D'=LJ+L+2C+2B$. For $t\ge D$, let $\mathcal A^t$ denote the set of all $x\in G$ such that $\sum_{\mathcal H_T(o,p)}[p,x]\ge t-D'$ and $d_{h\gamma}(p,x) \geq 2C$ and $d_{h\gamma}(o,x) \geq 2C$, where $h\gamma \in \mathcal H_T(o,p)$ is the minimal coset contributing to the sum $\sum_{\mathcal H_T(o,p)}[p,x]$, with respect to the linear order from Lemma \ref{lem:linearorder} for $\mathcal H_T(o,p)$. For $k \leq n$ and $x\in G$, let $\mathcal A^{t}_{k,x}$ denote the event 
\begin{displayquote}
``\emph{We have $w^{p}_k=x$ and $x\in\mathcal A^t$ and $\ \forall i<k \ :w_i^p \not\in \mathcal A^t$.}''
\end{displayquote}

That is, $\mathcal A^{t}_{k,x}$ is the event that $w_k^p=x$ is the first time that $w_k^p \in \mathcal A^t$.

\begin{claim}
\label{claim:definite_proba}
For any $k\leq n$, $x\in G$, and $t\geq D$, we have
$$
\mathbb P\left[\sum_{\mathcal{H}_T(o,p)} [p,w^p_n] \in  [t-D, t+D]  \,\Big\vert\, \mathcal A_{k,x}^t \right] \,\ge\, \mathbb P \big[d_{h_1\gamma}(x,w_{n-k}^x) \leq C \big]. 
$$
\end{claim}

\begin{subproof}{Proof of Claim}
We fix $k \le n$, $x\in G$ and $t\geq D$ and assume that $\mathcal A^t_{k,x}$ holds. Let $h\gamma$ be the minimal coset, with respect to the linear order from Lemma \ref{lem:linearorder} for the set $\mathcal H_T(o,p)$ that contributes to the sum $\sum_{\mathcal H_T(o,p)}[p,w_k^p]$. 

Therefore  \begin{equation}
\label{eqn:minimal_contributor}
    \pi_{h\gamma}(p) \neq \pi_{h\gamma}(w_k^p)\quad \text{and} \quad \pi_{h'\gamma}(p)=\pi_{h'\gamma}(w_k^p) \quad \forall h'\gamma \prec h\gamma.
\end{equation}

By the Behrstock inequality (Lemma~\ref{lem:Behrstock_for_gamma}), any point $y\in G$ with $\dist_{h\gamma}(p,y)>B$ has $\pi_{h''\gamma}(y)=\pi_{h''\gamma}(h\gamma)=\pi_{h''\gamma}(o)$ for all $h''\gamma\succ h\gamma$. Similarly, any point $y\in G$ with $\dist_{h\gamma}(o,y)>B$ has $\pi_{h'\gamma}(y)=\pi_{h'\gamma}(h\gamma)=\pi_{h'\gamma}(p)$ for all $h'\gamma\prec h\gamma$ (and in $\mathcal H_T(o,p)$). Hence, if $\dist_{h\gamma}(w_k^p,w^p_{n})\le C$, then $\sum_{\mathcal H_T(o,p)}[w_k^p,w^p_{n}]=\dist_{h\gamma}(w_k^p,w^p_{n})\le C$ as $w_k^p \in \mathcal A^t$. Hence if $\mathcal A^t_{k,x}$ holds and if $d_{h\gamma}(w_k^p, w_n^p) \leq C$ then 
\[
\sum_{\mathcal{H}_T(o,p)} [p,w^p_{n}] \,\geq 
    \sum_{\mathcal H_T(o,p)}[p,w_k^p \big] \,-\sum_{\mathcal H_T(o,p)}[w_k^p,w_{n}^p \big] 
    \,\ge\, t-D'-C \,\ge\, t-D,
\]
by the triangle inequality (Remark~\ref{rem:triangle_inequality}).

Now, by the triangular inequality (Remark \ref{rem:triangle_inequality}) we also have that if $d_{h\gamma}(w_k^p, w_n^p) \leq C$, then

$$\sum_{\mathcal H_T(o,p)}[p,w_n^p] \le \sum_{\mathcal H_T(o,p)}[p,w_k^p]+C.$$

Therefore, in order to complete the proof of the claim we need to bound $\sum_{\mathcal H_T(o,p)}[p,w_k^p]$.

\begin{subclaim}
\label{subclaim:not_that_big_sum}
    $\sum_{\mathcal H_T(o,p)}[p,w_k^p] \leq t-D'+4C+2J.$
\end{subclaim}

\textit{Proof of Subclaim \ref{subclaim:not_that_big_sum}} Let $h_1\gamma$ be the second minimal coset (with respect to the linear order from Lemma \ref{lem:linearorder} for $\mathcal H_T(o,p)$) contributing to the sum $\sum_{\mathcal H_T(o,p)}[p,w_k^p]$. Let $r < k$ be the last time such that $d_{h_1\gamma}(o,w_r^p)> 2C$. Then we have that $d_{h_1\gamma}(p,w_r^p) \geq d_{h_1\gamma}(o,p)-d_{h_1\gamma}(o,w_{r+1}^p)-d_{h_1\gamma}(w^p_{r+1},w_r^p) \geq T-2C-J >2C$. Therefore, $\sum_{\mathcal H_T(o,p)}[p,w_r^p]<t-D'$ otherwise $w_r^p \in \mathcal A^t$ contradicting the minimality of $k$. By the strong Behrstock inequality, we  have that $\pi_{h'''\gamma}(o)=\pi_{h'''\gamma}(w_r^p)=\pi_{h'''\gamma}(w_k^p)$ for all $h'''\gamma \succ h_1\gamma$. We also have that $d_{h\gamma}(o,w_r^p) >B$ and hence $\pi_{h'\gamma}(p)=\pi_{h'\gamma}(w_r^p)=\pi_{h'\gamma}(w_k^p)$ for all $h'\gamma \prec h\gamma$. As $d_{h_1\gamma}(o,w_r^p) >B$ and $h\gamma \prec h_1\gamma$ we have that $\pi_{h\gamma}(p)=\pi_{h\gamma}(w_r^p)$.

Therefore, 
\begin{align*}
    \begin{split}
        \sum_{\mathcal H_T(o,p)}[p,w_k^p] &\le \sum_{\mathcal H_T(o,p)}[p,w_r^p]+\sum_{\mathcal H_T(o,p)}[w_r^p,w_r^p] \\
        &\leq t-D'+d_{h\gamma}(w_r^p, w_k^p)+d_{h_1\gamma}(w_r^p, w_k^p) \\
        &\leq t-D'+d_{h\gamma}(w_r^p, w_k^p)+d_{h_1\gamma}(w_{r+1}^p, w_r^p)+d_{h_1\gamma}(w_{r+1}^p, w_k^p) \\
        &\leq t-D'+d_{h\gamma}(p, w_k^p)+J+2C.
    \end{split}
\end{align*}

By definition of $\mathcal A_{k,x}^t$, there is a last time $s<k$ such that $d_{h\gamma}(p,w_s^p) \le 2C$. If $s+1=k$ then $d_{h\gamma}(p, w_k^p) \leq d_{h\gamma}(p,w_s^p)+J \leq 2C+J$ and this proves the subclaim.

If $s<k-1$ then let $u\in (s+1, k)$ be the last time such that $\sum_{\mathcal H_T(o,p)}[p,w^p_u] <t-D'$. 
If $u+1=k$ then $$\sum_{\mathcal H_T(o,p)}[p,w^p_k] \leq \sum_{\mathcal H_T(o,p)}[p,w^p_u]+\sum_{\mathcal H_T(o,p)}[w^p_u,w^p_{u+1}] \leq t-D'+J $$ and we have proved the subclaim. If not, then we consider $w_{i}^p$, for $i=k-1$ or $i=u+1$ (these might be equal but that's fine), we have that $k>i >u$. By definition of $w_k^p$, we must have $w_{i}^p \not\in A^t$, but $i >\max\{s,u\}$ hence $\sum_{\mathcal H_T(o,p)}[p,w^p_{i}]\geq t-D'$ and $d_{h\gamma}(p,w_{i}^p) \geq 2C$. Therefore $d_{h\gamma}(o,w_{i}^p) <2C$ for $i=k-1$ or $i=u+1$. Now, by the fact that the projections of $w_k^p$ and $w_u^p$ coincide on cosets which are not $h\gamma$ and by the triangular inequality (Remark \ref{rem:triangle_inequality}), we get 
\begin{align*}
    \begin{split}
        \sum_{\mathcal H_T(o,p)}[p,w_k] & \le \sum_{\mathcal H_T(o,p)}[p,w^p_u]+\sum_{\mathcal H_T(o,p)}[w^p_u,w^p_k] \\
        &\leq t-D'+d_{h\gamma}(w^p_u,w^p_k) \\
       & \leq t-D'+d_{h\gamma}(w^p_u,w^p_{u+1})+d_{h\gamma}(w^p_{u+1},o)+d_{h\gamma}(o,w^p_{k-1})+d_{h\gamma}(w^p_{k-1},w^p_{k}) \\
        &< t-D'+J+2C+2C+J
    \end{split}
\end{align*}
and this proves the subclaim.
$\blacklozenge$

Therefore, by Subclaim \ref{subclaim:not_that_big_sum}, if $A_{k,x}^t$ holds then the event  "$d_{h\gamma}(w_k^p, w_{n-k}^p) \leq C$" implies that $$\sum_{\mathcal H_T(o,p)}[p,w^p_n] \leq  \sum_{\mathcal H_T(o,p)}[p,w^p_k]+\sum_{\mathcal H_T(o,p)}[w_k^p,w^p_n] \leq t-D'+4C+2J+C \leq t+D. $$

Combining this with the result obtained earlier (above the subclaim), we get that 
\begin{align*}
    \begin{split}
      \mathbb P \Big[ \sum_{\mathcal H_T(o,p)}[p,w_n^p] \in[t-D, t+D] \vert A_{k,x}^t\Big] &= \mathbb P \Big[ d_{h\gamma}(w_k^p, w_n^p) \le C \vert A_{k,x}^t\Big] \\
      &=\mathbb P \Big[ d_{h\gamma}(x, w_{n-k}^x) \le C\Big]
    \end{split}
\end{align*}
by the strong Markov property \cite[Lemma 2.2]{goldsboroughsisto:markov}.
\end{subproof}

Since the conclusion of Proposition~\ref{prop:small_proj} holds for $g$, with $\gamma=E(g)$, we have $\mathbb P[\dist_{h_1\gamma}(x,w^x_{n-k})\le C]>\epsilon$. Claim \ref{claim:definite_proba} above and the law of total probability therefore gives us
\begin{align*}
f(t-D)&-f(t+D) \,=\, \mathbb P\Big[\sum_{\mathcal{H}_T(o,p)} [p,w^p_n] \in  [t-D, t+D] \Big] \\
&\geq\, \sum_{k\le n}\sum_{x\in G}\Big( 
    \mathbb P\Big[\sum_{\mathcal{H}_T(o,p)} [p,w^p_n] \in  [t-D, t+D]  \Big\vert \mathcal A_{k,x}^t \big] 
    \cdot \mathbb P\big[\mathcal A_{k,x}^t \big]\Big) \\ 
&>\, \epsilon \sum_{k \leq n}\sum_{x \in G} \mathbb P \left[\mathcal A_{k,x}^t  \right] \\ 
&\ge\, \epsilon\mathbb P \big[\exists k \leq n : w_{k}^p \in \mathcal A^t \big].
\end{align*}
To complete the proof of the lemma, we show that $g(t) \leq \mathbb P \big[\exists k \leq n : w_{k}^p \in \mathcal A^t\big]$. For this, suppose that the defining event of $g(t)$ holds. That is, suppose that there exists $r\le n$ such that $S=\sum_{\mathcal H_T(o,p)}[p,w^p_r]\ge t$. Let $h\gamma,h'\gamma\in\mathcal H_T(o,p)$ be the minimal and second-minimal elements contributing to the sum $S$, respectively. Note that $\pi_{h''\gamma}(o)=\pi_{h''\gamma}(w^p_r)$ for every $h''\gamma\succ h\gamma$.

If $\dist_{h\gamma}(p,w^p_r)\ge2C$ and $\dist_{h\gamma}(o,w^p_r)\ge2C$, then $w^p_r\in\mathcal A^t$. In particular, there exists $k\le n$ such that $w^p_k\in\mathcal A^t $ and we are done. 

Otherwise, $\dist_{h\gamma}(p,w^p_r)<2C$ or $\dist_{h\gamma}(o,w^p_r)<2C$. First, say that $\dist_{h\gamma}(p,w^p_r)<2C$. Since $h\gamma$ contributes to the sum $S$, we have $\pi_{h\gamma}(p)\ne\pi_{h\gamma}(w^p_r)$, so by Lemma~\ref{lem:Behrstock_for_gamma} we must have $\dist_{h'\gamma}(o,w^p_r)\le B$, and in particular $\dist_{h'\gamma}(p,w^p_r)>B$. Because $T>2(C+B+J)$, there is a maximal $r'\leq r$ such that $\dist_{h'\gamma}(o,w^p_{r'})>2C>B$. 

By the choice of $J$, we have $\dist_{h'\gamma}(p,w^p_{r'})>2C>B$. By the assumption that $\dist_{h\gamma}(p,w^p_r)<2C$, we have $\dist_{h\gamma}(o,w^p_r)>B$. By applying Lemma~\ref{lem:behrtock_inequality} using these various estimates, we can use the definition of the linear order on $\mathcal H_T(o,p)$ to obtain:
\begin{itemize}
\item   $\pi_{h^+\gamma}(w^p_r)=\pi_{h^+\gamma}(o)$ for all $h^+\gamma\succ h'\gamma$;
\item   $\pi_{h^-\gamma}(w^p_{r'})=\pi_{h^-\gamma}(p)$ for all $h^-\gamma\prec h'\gamma$;
\item   $\pi_{h^+\gamma}(w^p_{r'})=\pi_{h^+\gamma}(o)$ for all $h^+\gamma\succ h'\gamma$;
\item   $\pi_{h''\gamma}(w^p_r)=\pi_{h''\gamma}(p)$ for all $h''\gamma\prec h\gamma$.
\end{itemize}
In particular, $h'\gamma$ is the minimal element of $\mathcal H_T(o,p)$ that contributes to $\sum_{\mathcal H_T(o,p)}[p,w^p_{r'}]$, and we have $\dist_{h'\gamma}(p,w^p_{r'})>2C$ and $\dist_{h'\gamma}(o,w^p_{r'})>2C$. Moreover, Remark~\ref{rem:triangle_inequality} lets us compute
\begin{align*}
\sum_{\mathcal H_T(o,p)}[p,w^p_{r'}] \,
    &\ge\, \sum_{\mathcal H_T(o,p)}[p,w^p_r] \,-\, \sum_{\mathcal H_T(o,p)}[w^p_r,w^p_{r'}] \\
&\ge\, t-\big(\dist_{h\gamma}(w^p_r,w^p_{r'})+\dist_{h'\gamma}(w^p_r,w^p_{r'})\big) \\
&\ge\, t-\dist_{h\gamma}(w^p_r,p)-\big(\dist_{h'\gamma}(w^p_r,o)+\dist_{h'\gamma}(o,w^p_{r'})\big) \\
&\ge\, t-2C-B-(B+J) \,\ge\, t-D'.
\end{align*}
We have shown that $w^p_{r'}\in\mathcal A^t$. Thus, whenever the defining event of $g(t)$ holds, there is some $k\le n$ such that $w^p_k\in\mathcal A^t$. 

If, instead, it is the case that $d_{h\gamma}(o,w_r^p) <2C $ then let $r'$ be the last time that $d_{h\gamma}(o,w_{r'}^p) \geq 2C$. Then we have that $d_{h\gamma}(o,w_{r'}^p) \leq 2C+J$ by the choice of $J$, hence $d_{h\gamma}(p,w_{r'}^p) \geq d_{h\gamma}(p,o)-d_{h\gamma}(o,w_{r'}^p) \geq 2C$ by the choice of $T$. Further, by Remark \ref{rem:triangle_inequality} we have $$ \sum_{\mathcal H_T(o,p)}[p,w_{r'}^p] \geq \sum_{\mathcal H_T(o,p)}[p,w_{r}^p]-d_{h\gamma}(w_r^p, w_{r'}^p) \geq t-(6C+J) \geq t-D'.$$
Thus, $w_{r'}^p \in \mathcal A^t $. Hence, we have shown that the defining event of $g(t)$ implies that there exists a $k \leq n$ such that $w_k^p \in \mathcal A^t$. This completes the proof of the lemma.
\end{proof}  

We have shown that the probability $g(t)$ is bounded above by a probability $f(t-D)$. In order to prove Proposition \ref{prop:axiom} it suffices to show that the function $f$ decays exponentially.

\begin{proof}[Proof of Proposition \ref{prop:axiom}]
By definition, $g(t) \geq f(t) \geq f(t+D)$ for all $t$. Therefore for $t\geq D$, Lemma~\ref{lem:bound_on_g(t)} gives us $f(t-D)-f(t+D) \geq \epsilon g(t) \geq \epsilon f(t+D)$. Hence $f(t+D) \leq \frac{f(t-D)}{1+\epsilon}$. For simplicity, set $\epsilon_2=\frac1{1+\epsilon}<1$. Rephrasing this, for every $s\ge0$ we have $f(s+2D)\le\epsilon_2f(s)$. If we write $s=2qD+r$ with $q\in\mathbb N$ and $r\in[0,2D)$, then iterating this estimate yields 
\[
f(s) \,\le\, \epsilon_2^qf(r) \,\le\, \epsilon_2^q \,=\, \epsilon_2^{\frac{s-r}{2D}} 
    \,\le\, \frac1{\epsilon_2}\epsilon_2^\frac s{2D} \,=\, (1+\epsilon)\epsilon_2^{\frac s{2D}}.
\]
Changing the base of the exponential completes the proof, with $C'=\frac{2D}{\log(1+\epsilon)}$.

\end{proof}

\section{Descent of quasi-isometries in HHSs} \label{sec:descent}

This section and the next one concern hierarchically hyperbolic spaces and groups.  For more detailed background on hierarchical hyperbolicity the reader is directed to any of several expository treatments; see e.g. \cite{sisto:what} for a detailed conceptual explanation of the definition or \cite[Part 2]{casalsruizhagenkazachkov:real} for a technical overview.  For present purposes, we refer the reader to \cite[Def. 1.1]{behrstockhagensisto:hierarchically:2} for the definition of a hierarchically hyperbolic space (HHS).  

We will sometimes require the following properties of an HHS $(X,\s)$:
\begin{itemize}
    \item $(X,\s)$ is \emph{normalised} if all maps $\pi_U:X\to\mathcal CU$ to the various hyperbolic spaces $\mathcal CU$ are uniformly coarsely surjective; see \cite[Rem. 1.3]{behrstockhagensisto:hierarchically:2} for why this can always be assumed.

    \item $(X,\s)$ has the \emph{bounded domain dichotomy} if there is a constant $B$ such that $\diam(\mathcal CU)\leq B$ for any $U\in\s$ with $\mathcal CU$ bounded.  For this and the next definition, see also \cite[Sec. 3]{abbottbehrstockdurham:largest}.

    \item $(X,\s)$ has \emph{unbounded products} if it has the bounded domain dichotomy and for all $U\in\mathfrak S\smallsetminus\{S\}$ such that $\C U$ is unbounded, there exists $V\in\s$ such that $U\orth V$ and $\C V$ is unbounded. (Recall that $S$ denotes the unique $\nest$--maximal element of $\s$.)
\end{itemize}

Note that every HHG satisfies the bounded domain dichotomy, as there are finitely many isometry classes of domains. Other hierarchical hyperbolicity notions will be used incidentally and we will refer to the relevant literature as we go.

Two key notions for us are \emph{standard product regions} and \emph{factored spaces}. The standard product region $P_U$ associated to a domain $U$ is a subspace of the HHS under consideration which is naturally quasi-isometric to a product $F_U\times E_U$, where moving in the $F_U$ (resp. $E_U$) factor only changes projections to domains nested into (resp. orthogonal to) $U$; see \cite[Sec. 5]{behrstockhagensisto:hierarchically:2} and \cite[Sec. 15]{casalsruizhagenkazachkov:real}. Similarly, there are product regions associated to collections of pairwise orthogonal domains (which is the coarse intersection of the various standard product regions). 

Regarding factored spaces, the idea is that starting with an HHS, we can cone off standard product regions and obtain another HHS, where the associated index set is obtained from the original index set by removing all indices ``relevant'' for the product regions we coned-off (and the rest of the structure, in particular the hyperbolic spaces, remains the same). In particular, the new HHS is somewhat ``simpler''. We now make this more precise.

Let $(X,\dist,\s)$ be an HHS and let $\V\subset\s$ be downwards-closed in the nesting poset $(\s,\nest)$. In \cite{behrstockhagensisto:asymptotic}, a metric $\dist^\V\le\dist$ is constructed on $X$ such that $(X,\dist^\V,\s\smallsetminus\V)$ is an HHS. We call $\dist^\V$ the \emph{factored metric} and $(X,\dist^\V)$ the \emph{factored space} of $(\dist,\V)$. HHS automorphisms descend to factored spaces \cite[Proposition 19.1]{casalsruizhagenkazachkov:real}.

\begin{defn}\label{defn:orth-graph}
Given an HHS $(X,\s)$, let $\Gamma_{\s}$ be the graph with a vertex for each element of $\s$ and an edge joining $U$ to $V$ whenever $U\bot V$ and both $\C U$ and $\C V$ are unbounded.  
\end{defn}

\begin{remark}[Arranging unbounded products]\label{rem:gamma-bdd}
If $(X,\s)$ has the bounded domain dichotomy, then $\Gamma_{\s}$ has an edge whenever $\C U$ and $\C V$ have diameter more than $B$ and $U\bot V$. We mainly use the bounded domain dichotomy indirectly, via the construction in \cite[Sec. 3]{abbottbehrstockdurham:largest}: if $(X,\s_0)$ is an HHS with the bounded domain dichotomy, then there is an HHS structure $(X,\s)$ with \emph{unbounded products}. 
\end{remark}

The following proposition says roughly that quasi-isometries of HHSs descend to quasi-isometries of the factored spaces where we cone-off the product regions with the maximal number of factors. A clique in a graph is said to be \emph{largest} if its cardinality agrees with the clique number of the graph.

\begin{proposition}[{\cite[Cor.~6.3]{behrstockhagensisto:quasiflats}}] \label{prop:qis_descend_once:biinfinite}
Let $(X,\dist,\s)$ be an  HHS with the bounded domain dichotomy, and let $\V$ be the downwards-closure of the union of all largest cliques of $\Gamma_\s$. If there is some $D$ such that for each $\nest$--maximal $U\in\V$, every pair of points in $F_U$ lies on some biinfinite $D$--quasigeodesic of $F_U$, then every quasiisometry of $(X,\dist)$ is a quasiisometry of $(X,\dist^\V)$.
\end{proposition}

After applying the above proposition, we have an HHS $(X,\dist^\V,\s\smallsetminus\V)$.  Since $\V$ contained all $U\in\s$ belonging to a largest clique in $\Gamma_{\s}$, the maximum size of a clique in $\Gamma_{\s\smallsetminus\V}$ is strictly smaller than in $\Gamma_{\s}$, while the bounded domain dichotomy persists since the cone-off construction does not modify the $\C U$.

Applying the proposition repeatedly, we therefore obtain HHSs where the maximal number of factors of a product regions decreases until there are no products with at least two factors left. The space obtained in this way is an HHS $\widehat X$ that is hyperbolic in view of \cite[Cor. 2.16]{behrstockhagensisto:quasiflats}.  Moreover, quasi-isometries of the original HHS descend to this HHS. 

However, this hyperbolic space is not necessarily the maximal hyperbolic space associated to the original HHS. This is explained more concretely in the following remark.

\begin{remark}\label{rem:factored-space-not-CS}
Let $S\in\s$ be the maximal element.  By \cite[Cor. 2.9]{behrstockhagensisto:asymptotic}, the maximal hyperbolic space $\C S$ is naturally quasi-isometric to the factored space $(X,\dist^{\s-\{S\}})$.  This is in general not naturally quasi-isometric to $\widehat X$, as the HHS structure of the latter contains more domains than just $S$.


To illustrate this, skipping all technical details but hopefully still conveying the picture, consider the group
$$G=(\mathbb Z^2\ast\mathbb Z)\times\mathbb Z.$$
Notice that $G$ has a splitting with two vertex groups, a copy of $\mathbb Z^3$ and a copy of $\mathbb Z$, and two edge groups with vertex group $\mathbb Z$, one edge being a loop while the other connects the two vertices. The first step of the procedure cones-off quasiflats of rank 3, resulting in a space quasi-isometric to the Bass-Serre tree $T$ for the aforementioned splitting; in particular this is already a hyperbolic space. The action on $T$ is not acylindrical, so $T$ is not the top-level hyperbolic space for any HHG structure on $G$.  

To illustrate this in a different way, we describe a factored space on the mapping class group that, whilst not giving exactly $\widehat X$, gives the idea of how it can differ from $\C S$, where $S$ is a closed connected oriented surface of genus at least 3. Namely, let $\U$ be the set of subsurfaces that are neither $S$ nor the complement of a curve. Then $\Gamma_{\s\smallsetminus\U}$ has no edges, so $(X,\dist^\U)$ is hyperbolic. But it is not naturally quasiisometric to $\C S$, and it has loxodromic isometries that are not pseudo-Anosovs of $S$, namely partial pseudo-Anosovs supported on the complement of a single curve. The iterative construction of $\widehat X$ gives an even ``bigger'' space than $(X,\dist^\U)$ as above, meaning that $\widehat X$ has an HHS structure containing more domains than $S$ and complements of single curves. For instance, complements of four-holed spheres are also part of the HHS structure for $\widehat X$ (four-holed spheres get ``removed'' at the first iteration of the construction, so after then their complements are never involved in quasiflats of dimension at least 2).

More abstractly, Figure~\ref{fig:coning} shows how applying the iterative procedure to $\s$ can result in $\widehat X$ having a nontrivial HHS structure.

\begin{figure}[ht]
\includegraphics[height=2cm,trim = 0mm 5mm 0mm 5mm]{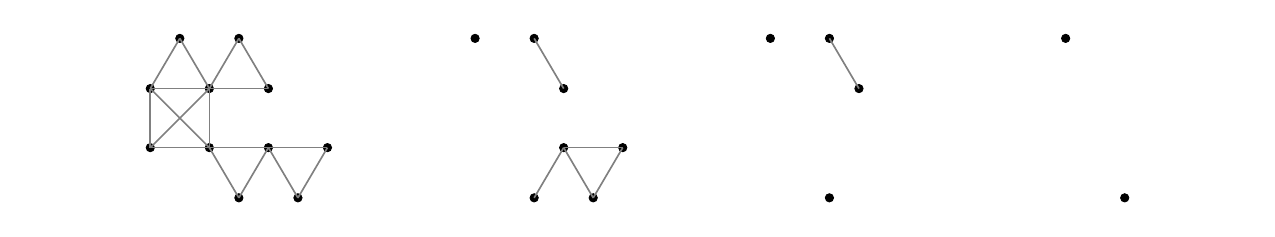}\centering
\caption{If the leftmost graph is part of $\Gamma_\s$ (with non-adjacent vertices transverse), then it yields the second graph in $\Gamma_{\s\smallsetminus\V}$, which is obtained by applying Proposition~\ref{prop:qis_descend_once:biinfinite}. Repeating yields the third graph, and we terminate with the fourth graph, which has vertices representing non-maximal unbounded domains. Hence $\widehat X$ is not quasiisometric to $\C S$ in this case.} \label{fig:coning}
\end{figure}
\end{remark}

We record the iterative construction described above here:

\begin{lemma} \label{lem:qi_descends_to_fake_curve_graph}
Let $(X,\dist,\s)$ be an HHS with the bounded domain dichotomy and the following property: if $F_U$ is unbounded then every pair of points of $F_U$ lies on a uniform biinfinite quasigeodesic of $F_U$. There is a downwards-closed set $\U\subset\s-\{S\}$ such that $(X,\dist^\U)$ is hyperbolic and every quasiisometry of $(X,\dist)$ is a quasiisometry of $(X,\dist^\U)$.
\end{lemma}

Note that, in view of the HHS structure on $(X,\dist^\U)$, the map $\pi_S:X\to \mathcal C S$ is still coarsely Lipschitz and coarsely surjective when $X$ is endowed with the metric $\dist^\U$.

\begin{proof}
Let $\omega$ be the clique number of $\Gamma_\s$. Let $\s^0=\s$, $\dist^0=\dist$. Given $\s^i$ and $\dist^i$, let $\cal U^i$ be the downwards-closure of the union of all maximal cliques of $\Gamma_{\s^i}$. Let $\dist^{i+1}$ be the factored metric of $(\dist^i,\U^i)$, and let $\s^{i+1}=\s^i\smallsetminus\U^i$. Because $(X,\dist^i,\s^i)$ is an HHS with the bounded domain dichotomy, so is $(X,\dist^{i+1},\s^{i+1})$. Moreover, the clique number of $\Gamma_{\s^i}$ is at most $\omega-i$. Thus there is some minimal $n<\omega$ such that $\Gamma_{\s^n}$ has no nontrivial cliques. Let $\cal U=\s\smallsetminus\s^n$. By Proposition~\ref{prop:qis_descend_once:biinfinite}, every quasiisometry of $(X,\dist)$ is a quasiisometry of every $(X,\dist^i)$, in particular of $(X,\dist^\U)=(X,\dist^n)$.
\end{proof}

From now on, $\U$, $\U^i$, and $\s^i$ denote the subsets of $\s$ constructed in proving Lemma~\ref{lem:qi_descends_to_fake_curve_graph}.

\begin{remark} \label{rem:sufficient_descend}
The proof of Lemma~\ref{lem:qi_descends_to_fake_curve_graph} did not need every $U\in\s$ to have the property that each pair of points in $F_U$ lies on a uniform quasigeodesic: it only needs this to hold for those $U$ that are $\nest$--maximal in some $\U^i$. Moreover, it does not need to hold for the full space $F_U$ obtained by considering $U\in\s$, only  the  $F^i_U$ obtained by considering $U\in\s^i$.
\end{remark}

\begin{lemma} \label{lem:CU_descent}
Let $(X,\s)$ be an HHS with unbounded products. If for all unbounded $U\in\s$, every $x\in\C U$ lies on a uniform biinfinite quasigeodesic, then every quasiisometry of $(X,\dist)$ is a quasiisometry of $(X,\dist^\U)$.
\end{lemma}

\begin{proof}
As noted in Remark~\ref{rem:sufficient_descend}, it suffices to show that if $U$ is $\nest$--maximal in some $\U^i$, then each pair of points of $F^i_U$ lies on some uniform quasigeodesic of $F^i_U$, where $F^i_U$ is the space of nested partial tuples in $(X,\dist^i,\s^i)$.

Let $U\in\U^i$ be $\nest$--maximal, and let $x_1,x_2\in F^i_U$. Consider the set $\beta=H_{\theta_0}(x_1,x_2)$. By the definition of $\U^i$, no two unbounded elements of $\s^i_U=\s^i\cap\s_U$ are orthogonal, so $\beta\subset F^i_U$ is a uniform quasigeodesic. Moreover, this implies that $F^i_U$ is hyperbolic, for instance by \cite[Thm~2.1]{bowditch:coarse}.

By \cite[Lem.~2.8]{behrstockhagensisto:hierarchically:2}, the set $\Rel_{100E}(x_1,x_2)\cap\s^i_U$ is totally ordered. Let $V_1$ and $V_2$ be the minimal and maximal elements, respectively. By assumption, there are uniform biinfinite quasigeodesics $\alpha_j\subset\C V_j$ through $\pi_{V_j}(x_j)$, respectively. Although the $\alpha_j$ might not be coarsely connected subsets of the hyperbolic space $F^i_U$, their convex hulls $\bar\alpha_j$ therein are uniform biinfinite quasigeodesics. By the distance formula \cite[Thm~4.5]{behrstockhagensisto:hierarchically:2} and the choice of the $V_j$, the quasigeodesic $\bar\alpha_j$ passes uniformly close to $x_j\in F^i_U$. We can therefore take subrays $\bar\alpha_j'$ such that $\bar\alpha_1'\cup\beta\cup\bar\alpha_2'$ is a uniform biinfinite quasigeodesic in $F^i_U$ through $x_1$ and $x_2$.
\end{proof}

The hyperbolic space $(X,\dist^{\mathcal U})$ to which quasi-isometries descend is set-theoretically unchanged from $X$, so still admits a map $\pi_S:X\to \C S$, where $S\in \s$ is the maximal element.  Moreover, as explained in the factored space construction in \cite{behrstockhagensisto:asymptotic}, this map is coarsely lipschitz.  However, it need not be a quasi-isometry, as explained in Remark~\ref{rem:factored-space-not-CS}.  We will need to further cone off $(X,\dist^{\U})$ to get a quasi-isometry to $\mathcal CS$.  

We now informally explain how to do this. We have to identify the (quasiconvex) subspaces that need to be coned off, which from the construction, come from product regions (with at least two factors) of the original HHS: they are isolated vertices as in Figure~\ref{fig:coning}. That is, in the original HHS there is a subspace of the form $U\times V$, and only one of the factors, say $V$, gets coned off when constructing $(X,\dist^{\mathcal U})$. To recognise those subspaces, we look at preimages of points in $(X,d^{\mathcal U})$ and notice that if two points lie in the remnant of a product region, their preimages have unbounded coarse intersection. In the notation above, preimages are of the form $\{u\}\times V$, and any two such subsets lie at finite Hausdorff distance (we stated bounded coarse intersection because we will take preimages of balls, not points, and we emphasise that the Hausdorff distance is finite but depends on the subsets in question). In fact, this characterises pairs of points in remnants of product regions, once this is properly quantified. We make this precise after setting notation.

If $(X,\s)$ is an HHS with $K$--unbounded products, then we can and shall assume that $K\ge\max\{1,K_0,K_1,K_2,K_3\}$, where the $K_i$ are the constants appearing in the statement of \cite[Lem.~1.20]{behrstockhagensisto:quasiflats}. Recall the function $\theta_u$ from the uniqueness axiom of HHSs: if $(X,\dist,\s)$ is an HHS and $\dist(x,y)\ge\theta_u(r)$, then there is some $U\in\s$ with $\dist_U(x,y)\ge r$. Also recall that every element $V\in\s\smallsetminus\{S\}$ has an orthogonal container $V^\bot$. Finally, recall (see \cite[Lem.~6.2]{behrstockhagensisto:hierarchically:2}) that for each HHS $(X,\s)$ there is a constant $\theta_0$ such that for any $A\subset X$, the $\theta_0$--hull $H_{\theta_0}A$, defined in \cite[Sec. 6]{behrstockhagensisto:hierarchically:2}, is hierarchically quasiconvex in the sense of \cite[Def.~5.1]{behrstockhagensisto:hierarchically:2} (meaning roughly that its projections in all hyperbolic spaces are quasiconvex, and it is coarsely maximal with those projections). We will also use the notion of a \emph{gate map} $\g_H$ to a hierarchically quasiconvex set $H$ \cite[Def.~5.4]{behrstockhagensisto:hierarchically:2}, defined by assembling all closest-point projections in the associated hyperbolic spaces.

\begin{proposition} \label{prop:product_characterisation}
Suppose that $(X,\dist,\s)$ is an HHS with $K$--unbounded products. There is a constant $C$ such that the following holds. If $U\not\in\U\cup\{S\}$ has $F_U$ unbounded, then every $x,y\in P_U$ lie in unbounded subsets $I_x,I_y\subset(X,\dist)$, respectively, that are at finite Hausdorff distance and have diameter at most $C$ in $(X,\dist^\U)$.

Conversely, for each $C$ there exists $D$ such that the following holds. Suppose that $x,y\in X$ lie in unbounded subsets $I_x,I_y\subset(X,\dist)$ that are at finite Hausdorff-distance and have diameter at most $C$ in $(X,\dist^\U)$. Either $\dist(x,y)\le2C+\theta_u(K)$, or there is some $V\in\U$ such that $x$ and $y$ are both $D$--close to $P_V$ in $(X,\dist^\U)$.
\end{proposition}

\begin{proof}
For the first statement, we can define $I_x=\{x\}\times E_U\subset P_U$ and $I_y=\{y\}\times E_U\subset P_U$. These are unbounded because $(X,\s)$ has unbounded products, and they are clearly at finite Hausdorff-distance. Moreover, the fact that $U\not\in\U$ implies that every unbounded $W\bot U$ is in $\U$. There is some such $W$ because $U\ne S$, so $I_x$ and $I_y$ have uniformly bounded diameter in $(X,\dist^\U)$.

Now let us consider the converse statement, in which we start with subsets $I_x$ and $I_y$ containing $x$ and $y$, respectively. Since $I_x$ and $I_y$ are at finite Hausdorff-distance, so are their projections to each $\C W$, and hence so are the convex hulls of these projections. This shows that the hierarchically quasiconvex hulls $H_{\theta_0} I_x$ and $H_{\theta_0} I_y$ are at finite Hausdorff-distance, and it follows that the gates $\g_{H_{\theta_0} I_x}I_y$ and $\g_{H_{\theta_0} I_y}I_x$ are unbounded. Let $I'_x=\g_{H_{\theta_0} I_x}I_y$ and let $I'_y=\g_{H_{\theta_0} I_y}I_x$.

Because $I'_x$ is unbounded, the uniqueness axiom implies the existence of a pair of points $a,a'\in I'_x$ with $\dist(a,a')\le2\theta_u(K)$ and a domain $V\in\Rel_{K}(a,a')$ such that $\Rel_K(a,a')\subset\s_V$. Consider $b=\g_{H_{\theta_0} I_y}a$. If $\dist(a,b)\le\theta_u(K)$, then 
\[
\dist^\U(x,y) \,\le\, \dist^\U(x,a)+\dist^\U(a,b)+\dist^\U(b,y) \,\le\, 2C+\theta_u(K).
\]

Otherwise, the set $\Rel_K(a,b)$ is nonempty. According to \cite[Lem.~1.20(5)]{behrstockhagensisto:quasiflats}, we have $U\bot V$ for all $U\in\Rel_K(a,b)$. Now suppose that $W\in\s$ has $V\pnest W$ or $V\trans W$. Using the fact that $\dist_V(a,a')\ge K$, consistency and bounded geodesic image tell us that at least one of $\pi_W(a),\pi_W(a')$ is $E$--close to $\rho^V_W$. This shows that $\dist(a,P_V)$ is bounded by some uniform constant $D'$ defined in terms of $\dist(a,a')\le2\theta_u(K)$ and the uniqueness function. Because of \cite[Lem~1.20]{behrstockhagensisto:quasiflats}, we can use a similar argument with slightly different constants to obtain a uniform bound $D''$ on $\dist(b,P_V)$ in terms of $\dist(b,\g_{H_{\theta_0}I_y}a')$, which is itself bounded in terms of $\dist(a,a')\le2\theta_u(K)$. We can now compute
\[
\dist^\U(x,P_V) \,\le\, \dist^\U(x,a) + \dist^\U(a,P_V) \,\le\, C+D',
\]
and $\dist^\U(y,P_V)$ is bounded similarly. Since $V\in\Rel_K(a,a')$, it is unbounded, and the condition that $\diam_{(X,\dist^\U)}I_x\le C$ implies that $V\in\U$. 
\end{proof}

We can now prove the main theorem of this section, after collecting all required hypotheses in the following definition.

\begin{defn}\label{defn:well-behaved-HHS}
    We call an HHS $(X,\s)$ \emph{well behaved} if it is normalised, has unbounded products, and any one of the following holds for all $U\in\s$:
    \begin{itemize}
\item   If $F_U$ is unbounded then each pair of points of $F_U$ lies on a uniform biinfinite quasigeodesic of $F_U$.
\item   If $\C U$ is unbounded, each $p\in \C U$ lies on a uniform biinfinite quasigeodesic of $\C U$.
\item   $\C U$ has uniformly cobounded isometry group.
\end{itemize}    
\end{defn}

\begin{theorem}\label{thm:qi-descend}
Let $(X,\s)$ be a well-behaved HHS. Then every quasiisometry $f$ of $X$ induces a quasiisometry $f'$ of $\C S$ such that $f'\pi_S$ and $\pi_Sf$ coarsely agree.
\end{theorem}

\begin{proof}
Let $f$ be a quasiisometry of $(X,\dist)$. Under the first assumption, Lemma~\ref{lem:qi_descends_to_fake_curve_graph} shows that $f$ is a quasiisometry of $(X,\dist^\U)$. Under the second assumption, this is given by Lemma~\ref{lem:CU_descent}, and the third assumption implies the second. 

Let $\V=(\s\smallsetminus\U)\smallsetminus\{S\}$. By the distance formula, the map $\pi_S$ is a quasiisometry from the HHS $(X,(\dist^\U)^\V,\{S\})$ to $\C S$ (here we use that the HHS is normalised, otherwise the map would only be a quasi-isometric embedding). It therefore suffices to show that $f$ is a coarsely lipschitz map of $(X,(\dist^\U)^\V)$, for then the same will apply to a quasiinverse of $f$. 

Let us write $\dist'=(\dist^\U)^\V$. By the definition of $\dist'$, we just have to check that if $x$ and $y$ both lie in some $F_U\subset(X,\dist^\U,\s\smallsetminus\U)$ with $U\in\V$ then $\dist'(f(x),f(y))$ is uniformly bounded. If $\diam_{(X,\dist^\U)}(F_U)<\infty$, then $\dist^\U(x,y)$ is uniformly bounded by the bounded domain dichotomy, and hence so is $\dist'(f(x),f(y))$, because $f$ is a quasiisometry of $(X,\dist^\U)$ and $\dist'\le\dist^\U$. 

Otherwise, the forward direction of Proposition~\ref{prop:product_characterisation} provides a pair of unbounded subspaces $I_x,I_y\subset(X,\dist)$ at finite Hausdorff distance that lie in bounded $\dist^\U$--neighbourhoods of $x$ and $y$, respectively. These properties are preserved by quasiisometries, so applying the reverse direction of Proposition~\ref{prop:product_characterisation} with $fI_x$ and $fI_y$ tells us that either $\dist(f(x),f(y))$ is uniformly bounded or there is a domain $V\in\s\smallsetminus\{S\}$ such that $f(x)$ and $f(y)$ are uniformly close to $P_V$ with respect to $\dist^\U$. In the latter case, the definition of $\dist'$ gives a uniform bound on $\dist'(x,y)$.
\end{proof}

In the case of hierarchically hyperbolic groups, we can remove the unbounded product assumption using results from \cite{abbottbehrstockdurham:largest}, at the expense of working in the slightly more general category of $G$--HHSes; see Definition \ref{defn:HHG}.

\begin{corollary}\label{cor:HHG-QI-descend-propaganda}
Let $(G,\s)$ be an HHG and suppose that the action of $\Stab_G(U)$ on $\C U$ is cobounded for every $U\in\s$, and that $G$ is not wide.  Then $G$ admits a $G$--HHS structure $(G,\s')$ such that the following holds, where $S\in\s'$ is the unique $\nest$--maximal element: $\mathcal CS$ is unbounded and every quasiisometry $f$ of $G$ induces a quasiisometry $f'$ of $\C S$ such that $f'\pi_S$ and $\pi_Sf$ coarsely agree.
\end{corollary}

\begin{proof}
This follows from Theorem \ref{thm:qi-descend} together with Lemma \ref{lem:making-good-behaviour} below (which follows from \cite{abbottbehrstockdurham:largest}), which can be used to remove the unbounded products hypothesis, and Lemma \ref{lem:irreducible}, which relates non-wideness to unboundedness of $\mathcal CS$.
\end{proof}

The conclusion of Theorem~\ref{thm:qi-descend} is not true without some extra hypothesis beyond unbounded products, as illustrated by the following example.

\begin{example}
Let $X$ be the universal cover of the plane minus an open square, which is a CAT(0) cube complex whose hyperclosure (see \cite{hagensusse:onhierarchical}) $\s$ is a factor system in the sense of \cite{behrstockhagensisto:hierarchically:1}, so that $(X,\s)$ is an HHS with unbounded products. However, there are self-quasiisometries of $X$, modelled on the logarithmic spiral map, that do not induce quasiisometries of the contact graph of $X$. See the discussion in \cite[\S3]{cashen:quasiisometries}.
\end{example}

\section{Acylindrical hyperbolicity and quasi-isometries}\label{sec:acyl-qi}
Throughout this section, we still work with a fixed finitely generated group $G$ and word metric $\dist_G$.  We start with a stronger version of Assumption \ref{assumpt:space}, which strengthens \emph{partial detectability} to full Morse detectability and the WPD property to acylindricity.

\begin{definition}[Geometrically faithful pair]\label{defn:strong-21}
Let $G$ act on a $\delta$--hyperbolic geodesic space $(X,\dist_X)$.  Fix a basepoint $x_0\in X$ and let $\rho:G\to X$ be the corresponding orbit map.  Assume that the action is cobounded, i.e. $X\subseteq N_\delta(\rho(G))$. Suppose that the following conditions all hold:
\begin{enumerate}
    \item\label{item:WPD-strong21} (\textbf{Acylindricity.}) $G$ acts on $X$ acylindrically and non-elementarily.
    \item \label{item:descend-strong21} (\textbf{QIs descend.}) For each $\nu$ there exists $\lambda$ such that for each $\nu$--quasi-isometry $\phi:G\to G$ there exists a $\lambda$--quasi-isometry $\bar \phi:X\to X$ such that $\dist_X(\phi(g)(x),\bar\phi(gx))\leq\lambda$ for all $g\in G$ and $x\in X$.

    \item \label{item:detect-strong21} (\textbf{Morse detectable.}) For every Morse gauge $M$ there exists $\lambda$ such that if $\gamma\subset G$ is an $M$--Morse geodesic, then $\rho\circ\gamma$ is a $\lambda$--quasigeodesic.  Conversely, for each $\lambda$ there is a Morse gauge $M$ such that if $\gamma\subset G$ is a geodesic such that $\rho\circ\gamma$ is a $\lambda$--quasigeodesic, then $\gamma$ is $M$--Morse.
\end{enumerate}
Then the pair $(G,X)$ is \emph{\needaname}.  We refer to the constant $\delta$, the map $\nu\mapsto \lambda$ implicit in condition \eqref{item:descend-strong21}, and the maps $\lambda\mapsto M$ and $M\mapsto \lambda$ implicit in condition \eqref{item:detect-strong21} as the \emph{parameters} of the \needaname pair $(G,X)$.
\end{definition}

A related coarse version (targeted at our applications) is as follows:

\begin{definition}[Quasi-\needaname pair]\label{defn:quasi-21}
Let $(X,\dist_X)$ be a $\delta$--hyperbolic space and let $\rho:(G,\dist_G)\to(X,\dist_X)$ be a $\delta$--coarsely surjective $\delta$--coarsely lipschitz map.  Suppose that the following all hold:
\begin{enumerate}
    \item \label{item:WPD-quasi21} (\textbf{Morse ray.}) $G$ contains a Morse geodesic ray.
    \item\label{item:descend-quasi21} (\textbf{QIs descend.}) For each $\nu$ there exists $\lambda$ such that for each $\nu$--quasi-isometry $\phi:G\to G$ there exists a $\lambda$--quasi-isometry $\bar \phi:X\to X$ such that $\dist_X(\rho(\phi(g)),\bar\phi(\rho(g)))\leq\lambda$ for all $g\in G$ and $x\in X$.
    \item \label{item:detect-quasi21} (\textbf{Morse detectable.}) For every Morse gauge $M$ there exists $\lambda$ such that if $\gamma\subset G$ is an $M$--Morse geodesic, then $\rho\circ\gamma$ is a $\lambda$--quasigeodesic. Conversely, for each $\lambda$ there is a Morse gauge $M$ such that if $\gamma\subset G$ is a geodesic such that $\rho\circ\gamma$ is a $\lambda$--quasigeodesic, then $\gamma$ is $M$--Morse.

    \item \label{item:geom-sep-quasi21} (\textbf{Geometrically separated fibres.})  For all $r\geq 0$ there exists $R\geq 0$ such that for all $x,y\in X$ such that $\dist_X(x,y)\geq R$, $$\diam\big(N^G_s(\rho^{-1}(N^X_r(x)))\cap \rho^{-1}(N_r^X(y))\big)<\infty$$
    for all $s\geq 0$, where $N^G$ and $N^X$ respectively denote neighbourhoods in $G$ and $X$.
\end{enumerate}

Then the pair $(G,X)$ is \emph{quasi-geometrically faithful}.  We refer to the constant $\delta$, the map $\nu\mapsto \lambda$ implicit in condition \eqref{item:descend-quasi21}, and the maps $\lambda\mapsto M$ and $M\mapsto \lambda$ implicit in condition \eqref{item:detect-quasi21} as the \emph{parameters} of the quasi-\needaname pair $(G,X)$.
\end{definition}

We relate the definitions, and show that Definition \ref{defn:quasi-21} is quasi-isometry invariant:

\begin{lemma}\label{lem:quasi-21-qi}
Let $H$ be a finitely generated group such that there exists a quasi-isometry $\psi:H\to G$, and suppose that $(G,X)$ is a quasi-\needaname pair, with $\rho:G\to X$ the coarsely lipschitz map from Definition \ref{defn:quasi-21}.  Then $(H,X)$ is a quasi-\needaname pair, with map $\rho\circ \psi$ and parameters depending only on the parameters of $(G,X)$ and the quasi-isometry constants of $\psi$.

Also, if $(G,X)$ is a \needaname pair, then it is a quasi-geometrically faithful pair.
\end{lemma}

\begin{proof}
Let $(G,X)$ be a quasi-\needaname pair, with $\rho$ as in the statement, and let $\psi:H\to G$ be a quasi-isometry.  Hence $\rho\psi:H\to X$ is a coarsely surjective coarsely lipschitz map.  Existence of a Morse ray is a quasi-isometry invariant property, so Definition \ref{defn:quasi-21}.\eqref{item:WPD-quasi21} holds for $H$. 

Suppose that $\phi:H\to H$ is a quasi-isometry.  Then by Definition~\ref{defn:quasi-21}.\eqref{item:descend-quasi21} we have a coarsely commutative diagram
\begin{center}
    $
    \begin{diagram}
     \node{H}\arrow{s,l}{\phi}\arrow{e,t}{\psi}\node{G}\arrow{s,l}{\phi'}\arrow{e,t}{\rho}\node{X}\arrow{s,l}{\bar\phi'}\\
     \node{H}\arrow{e,b}{\psi}\node{G}\arrow{e,b}{\rho}\node{X}
    \end{diagram}
    $
\end{center}
(constants depending only on the parameters of $(G,X)$ and the quasi-isometry constants for $\phi$), where $\phi'=\psi\phi\hat\psi$ and $\hat\psi$ is a quasi-inverse of $\psi$. By construction, using $\rho\psi$ in place of $\rho$, we have verified Definition \ref{defn:quasi-21}.\eqref{item:descend-quasi21} for $(H,X)$.

Next, we check Definition \ref{defn:quasi-21}.\eqref{item:detect-quasi21} for $(H,X)$, using the map $\rho\psi$.  This follows since the corresponding property holds for $(G,X)$, and in fact the same property, with geodesics replaced by quasigeodesics, also holds in $(G,X)$, using e.g. \cite[Lem. 2.8]{russellsprianotran:local}, and we can move quasigeodesics between $G$ and $H$ using $\psi$.

Finally, we verify the geometric separation for fibers, i.e. that $(H,X)$ satisfies Definition \ref{defn:quasi-21}.\eqref{item:geom-sep-quasi21}.  Let $s\geq 0$ and let $x,y\in X$. Fix $a,b,a',b'\in H$ such that $\dist_H(a,a'),\dist_H(b,b')\leq s$ and $\dist_X(\rho\psi(a),x),\dist_X(\rho\psi(b),x)\leq r$ and $\dist_X(\rho\psi(a'),y),\dist_X(\rho\psi(b'),y)\leq r$.  Then $\dist_G(\psi(a),\psi(a')),\dist_G(\psi(b),\psi(b'))$ are bounded in terms of $s$ and $\psi$, and so Definition \ref{defn:quasi-21}.\eqref{item:geom-sep-quasi21}, applied to $(G,X)$, bounds $\dist_G(\psi(a'),\psi(b'))$ (the bound is allowed to depend on $x,y,s$), and we therefore get the required bound on $\dist_H(a',b')$.

Now we prove the second assertion, that \needaname implies quasi-geometrically faithful. Suppose that $(G,X)$ is a \needaname pair, with orbit map $\rho$.  Then Definition \ref{defn:quasi-21}.\eqref{item:descend-quasi21},\eqref{item:detect-quasi21} are immediate from Definition \ref{defn:strong-21}.\eqref{item:descend-strong21},\eqref{item:detect-strong21}.  Let $g\in G$ be a loxodromic WPD element, so that $n\mapsto \rho(g^n)x_0=\rho(g^n)$ is a quasi-isometric embedding $\mathbb Z\to X$, and hence $n\mapsto g^n$ is a Morse quasigeodesic $\mathbb Z\to G$, because of Definition \ref{defn:strong-21}.\eqref{item:detect-strong21}.  Restricting to $\mathbb N$ gives a Morse ray in $G$, so Definition \ref{defn:quasi-21}.\eqref{item:WPD-quasi21} holds.  Definition \ref{defn:quasi-21}.\eqref{item:geom-sep-quasi21} follows from Definition \ref{defn:strong-21}.\eqref{item:WPD-strong21} along with \cite[Lem. 3.3]{sisto:quasi-convexity}.  Indeed, the acylindricity assumption guarantees that $G$ acts acylindrically along $X$ in the sense of \cite[Def. 3.1]{sisto:quasi-convexity} (viewing $X$ as a subspace of itself), and the property of preimages of balls mentioned in Definition \ref{defn:quasi-21} coincides with geometric separation in the sense of \cite[Def. 2.1]{sisto:quasi-convexity}.
\end{proof}

The main technical result of this section is:

\begin{proposition}[Acylindrical hyperbolicity from quasi-\needaname pair]\label{prop:quasi-implies-strong}
Let $(G,X)$ be a quasi-\needaname pair.  Then there exists a hyperbolic geodesic space $Y$ such that $G$ acts on $Y$, there exists $g\in G$ acting on $Y$ as a loxodromic element, and every loxodromic $g\in G$ is a WPD element for the $G$--action on $Y$.  In particular, if $G$ is nonelementary, then $G$ is acylindrically hyperbolic.
\end{proposition}

Before the proof, we note:

\begin{corollary}\label{cor:acyl-qi}
Let $G$ be a nonelementary finitely generated group such that there is a hyperbolic space $X$ making $(G,X)$ a \needaname pair, and let $H$ be a finitely generated group quasi-isometric to $G$.  Then $H$ is acylindrically hyperbolic. 
\end{corollary}

\begin{proof}
By Lemma \ref{lem:quasi-21-qi}, $(G,X)$, and hence $(H,X)$, is a quasi-\needaname pair, so applying Proposition \ref{prop:quasi-implies-strong} implies that $H$ is acylindrically hyperbolic.
\end{proof}

We now prove the proposition:

\begin{proof}[Proof of Proposition \ref{prop:quasi-implies-strong}]
Fix a quasi-\needaname pair $(G,X)$ with map $\rho$ and parameters as in Definition \ref{defn:quasi-21}.

Once we produce a hyperbolic space $Y$ on which $G$ acts with a loxodromic WPD element, then, under the additional assumption that $G$ is nonelementary, acylindrical hyperbolicity follows from \cite[Thm. 1.2]{osin:acylindrically}.  In particular, the conclusion that every loxodromic $g\in G$ is WPD is stronger than necessary, but we included it in the statement since it introduces no extra complexity to the proof.

\textbf{Quasi-action of $G$ on $X$:}  Each $g\in G$ can be regarded as a left-multiplication isometry $g:G\to G$, so by Definition \ref{defn:quasi-21}.\eqref{item:descend-quasi21}, there is a constant $\lambda$ such that for all $g$, we have a $\lambda$--quasi-isometry $A(g):X\to X$ such that $\dist_X(A(g)(\rho(x)), \rho(gx))\leq \lambda$ for all $x\in G$.  Up to uniformly increasing $\lambda$, this implies that $\dist_X(A(gh)(y),A(g)(A(h)(y)))\leq \lambda$ for all $y\in X$ and $g,h\in G$.  Here we have used coarse surjectivity of $\rho$.  

Moreover, for any $y\in X$, we can choose $x\in G$ with $\dist_X(\rho(x),y)\leq \lambda$, and then $\dist_X(A(x^{-1})(\rho(x)),\rho(1))\leq \lambda$, so since $A(x^{-1})$ is uniformly coarsely lipschitz, we have, up to uniformly enlarging $\lambda$, that $\dist_X(y, A(g)(\rho(1)))\leq \lambda$ for some $g\in G$.  Thus far, we have produced a uniform $\lambda$ such that $A:G\to X^X$ is a $\lambda$--cobounded $\lambda$--quasi-action of $G$ on $X$ by $\lambda$--quasi-isometries.

\textbf{The space $Y$:}  Using any of the various ``Milnor-\v{S}varc for quasi-actions'' statements in the literature, we now replace the quasi-action by an action.  For instance, \cite[Prop. 4.4]{manning:pseudocharacters} provides, up to uniform enlargement of $\delta$ and $\lambda$, a $\delta$--hyperbolic graph $Y$, a $\lambda$--quasi-isometry $q:X\to Y$, and a homomorphism $C:G\to \mathrm{Isom}(Y)$ such that 
\begin{itemize}
    \item $C$ is $\lambda$--cobounded, and
    \item $\sup_{x\in X}\dist_Y(C(g)(q(x)),q(A(g)(x)))\leq\lambda$ for all $g\in G$.
\end{itemize}

Since the action is cobounded, in particular the orbits are quasiconvex and from, e.g., \cite[Prop. 3.2]{capracecornuliermonodtessera:amenable}, there exists $g\in G$ acting on $Y$ loxodromically, because $Y$ is unbounded in view of Definition \ref{defn:quasi-21}.\eqref{item:WPD-quasi21},\eqref{item:detect-quasi21}.  Let $\tau:G\to Y$ be $\tau(g)=C(g)(q(\rho(1)))$, which $\lambda$--coarsely coincides with $q(A(g)(\rho(1)))$ and so with $q(\rho(g))$.

\textbf{Checking WPD:}  We now verify that $g$ is a WPD element.  In our notation, this means we must show that for each $\epsilon>0$, there exists $n\in\mathbb Z$ such that $|\mathcal H(n,\epsilon)|<\infty$, where 
$$\mathcal H(n,\epsilon)=\left\{h \in G:\dist_Y(C(h)(\tau(1)), \tau(1))<\epsilon, \dist_Y(C(h)(\tau(g^n)),\tau(g^n))<\epsilon\right\}.$$
There exists $r$, depending only on $\epsilon$ and $q$, but not on $n$, such that $h\in \mathcal H(n,\epsilon)$ implies $\dist_X(\rho(1),A(h)(\rho(1)))\leq r$ and $\dist_X(\rho(g^n),A(h)(\rho(g^n)))\leq r$.  So by the triangle inequality,
$$\dist_X(\rho(1),\rho(h))\leq r+\lambda$$
and 
$$\dist_X(\rho(g^n),\rho(hg^n))\leq r+\lambda.$$
Definition \ref{defn:quasi-21}.\eqref{item:geom-sep-quasi21} provides an $R=R(r+\lambda)$ such that $\dist_X(\rho(1),\rho(g^n))\geq R$ implies that the $\rho$--preimages of $N^X_{r+\lambda}(\rho(1))$ and $N^X_{r+\lambda}(\rho(g^n))$ are geometrically separated.

For $n\in\mathbb N$, let $\gamma_n$ be a geodesic in $G$ joining $1$ to $g^n$.  Since $g$ is loxodromic, the composition $\tau \gamma_n$ is a quasi-isometric embedding with constants depending on $\lambda,\delta,$ and $g$ but not on $n$.  Hence the map $\rho\gamma_n$ is also a quasi-isometric embedding with constants independent of $n$, so for sufficiently large $n$ we have $\dist_X(\rho(1),\rho(g^n))>R$.  Fix such an $n$.  

If $h\in\mathcal H(n,\epsilon)$, then from the earlier discussion we have $h\in \rho^{-1}(N_{r+\lambda}^X(\rho(1)))$ and $hg^n\in \rho^{-1}(N_{r+\lambda}^X(\rho(g^n)))$.  Thus $h\in N^G_s(\rho^{-1}(N^X_{r+\lambda}(\rho(g^n))),$ where $s=\dist_G(1,g^n)$ is independent of $h$.  So Definition \ref{defn:quasi-21}.\eqref{item:geom-sep-quasi21} allows only finitely many possibilities for $h$, as required.
\end{proof}

\subsection{Application to hierarchically hyperbolic groups}\label{subsec:HHG-application-quasi21}
We now use Theorem \ref{thm:qi-descend} to apply Proposition \ref{prop:quasi-implies-strong} to hierarchically hyperbolic groups.  We first recall the definition of an HHG (see e.g. \cite{petytspriano:unbounded} for the following modern formulation and \cite{durhamhagensisto:correction} for why it is equivalent to the original definition from \cite{behrstockhagensisto:hierarchically:2}):

\begin{defn}[Hierarchically hyperbolic group, $G$--HHS]\label{defn:HHG}
Let $G$ be a finitely generated group.  Suppose that $G$, with the quasi-isometry class of word metrics associated to finite generating sets, admits an HHS structure $(G,\mathfrak S)$.  Suppose, moreover, that $G$ acts cofinitely on $\mathfrak S$ preserving $\nest,\orth,\transverse$, and that the following hold for all $U,V\in \mathfrak S$:
\begin{itemize}
 \item for each $g\in G$, there is an isometry $g:\mathcal CU\to\mathcal CgU$ such that for all $g,h\in G$,
 \item the composition $\mathcal CU\stackrel{h}{\longrightarrow}\mathcal ChU\stackrel{g}{\longrightarrow}\mathcal CghU$ agrees with the isometry $gh$, and 
 \item $\pi_{gU}(gx)=g(\pi_U(x))$ for all $x\in G$ and 
 \item $\rho^{gU}_{gV}=g(\rho^U_V)$ whenever $U\transverse V$ or $U \propnest V$.
\end{itemize}
Then $(G,\mathfrak S)$ is a \emph{hierarchically hyperbolic group}.

Following \cite[Def. 3.3]{abbottbehrstockrussell:structure}, we say the pair $(G,\s)$ is a \emph{$G$--HHS} if it has all of the above properties, except weakened in the following way: we \emph{do not} require $\s$ to contain only finitely many $G$--orbits, but we do require $(G,\s)$ to satisfy the bounded domain dichotomy.
\end{defn}

\begin{remark}\label{rem:GHHS}
The reason to work with $G$--HHSes is twofold: first, various results about HHGs  hold in the slightly more general context of $G$-HHSes, and second, the \emph{maximisation} procedure from \cite{abbottbehrstockdurham:largest}, used to produce an HHS with the bounded domain dichotomy into one with unbounded products, will in general convert an HHG structure on $G$ into a $G$--HHS structure that does not necessarily have finitely many orbits.  See \cite[Rem. 3.4]{abbottbehrstockrussell:structure}.
\end{remark}

Our main result is about acylindrically hyperbolic HHGs where the top-level hyperbolic space witnesses acylindrical hyperbolicity.  We formulate this is follows:

\begin{defn}[Irreducible]\label{defn:unbounded-CS}
The $G$--HHS $(G,\mathfrak S)$ is \emph{irreducible} if $G$ has unbounded orbits in $\mathcal CS$, where $S\in\mathfrak S$ is the unique $\nest$--maximal element.
\end{defn}

Irreducibility is related to acylindrical hyperbolicity, and various other properties by assembling various results in the literature:

\begin{lemma}[HHG irreducibility criteria]\label{lem:irreducible}
The following are equivalent, for a normalised $G$--HHS $(G,\mathfrak S)$ with unbounded products and $G$ nonelementary:
\begin{enumerate}
    \item $(G,\mathfrak S)$ is irreducible.\label{item:irred-HHG}
      \item $G$ is acylindrically hyperbolic.\label{item:acyl-HHG}
       \item $G$ has a Morse element.\label{item:Morse-HHG}

    \item $G$ is not \emph{wide}, i.e. asymptotic cones of $G$ have cut-points.\label{item:nonwide-HHG}
      
    \item $G$ is not quasi-isometric to the product of two unbounded metric spaces.\label{item:nonproduct-HHG}

    \item $G$ has no nonempty finite orbit in $\mathfrak S-\{S\}$.\label{item:finite-orbit-HHG}
\end{enumerate}
\end{lemma}

\begin{proof}
Since the action of $G$ on $\mathcal CS$ is always acylindrical \cite[Thm. 14.3]{behrstockhagensisto:hierarchically:1}, and we are assuming $G$ is nonelementary, \eqref{item:irred-HHG} implies \eqref{item:acyl-HHG}.  Item \eqref{item:acyl-HHG} implies \eqref{item:Morse-HHG} and \eqref{item:nonwide-HHG} by \cite{sisto:quasi-convexity}, and both of the latter imply \eqref{item:nonproduct-HHG}.  If there exists $U\in\mathfrak S-\{S\}$ such that $G\cdot U$ is finite, then $G$ virtually stabilises the standard product region $P_U$ which, by the unbounded products assumption, is quasi-isometric to the product of two unbounded metric spaces (see \cite[Sec. 15]{casalsruizhagenkazachkov:real} for the quasi-isometry computation).  Hence \eqref{item:nonproduct-HHG} implies \eqref{item:finite-orbit-HHG}.  Finally, using the normalisation assumption, \cite[Cor. 9.9]{durhamhagensisto:boundaries} says that \eqref{item:finite-orbit-HHG} implies \eqref{item:irred-HHG}.  The latter implication also follows from \cite[Thm. 5.1]{petytspriano:unbounded}.
\end{proof}

For what follows, in view of results from \cite{abbottbehrstockdurham:largest} we can work with more general HHSs than well-behaved ones. Essentially, we can drop the unbounded products assumption:

\begin{definition}[Reasonably behaved]\label{defn:well-behaved}
The HHS $(X,\s)$ is \emph{reasonably behaved} if it is normalised, it satisfies the \emph{bounded domain dichotomy} from \cite[Def. 3.2]{abbottbehrstockdurham:largest} and for all $U\in\mathfrak S$, $\Isom(\mathcal CU)$ acts on $\mathcal CU$ coboundedly. 
\end{definition}

The next lemma is an application of results in \cite{abbottbehrstockdurham:largest}, used to avoid having to add the hypothesis of \emph{unbounded products} to the final statements:

\begin{lemma}\label{lem:making-good-behaviour}
Let $(G,\s)$ be a reasonably behaved irreducible $G$--HHS. Then there is a $G$--HHS structure $(G,\mathfrak T)$ that is well behaved and irreducible.
\end{lemma}

\begin{proof}
Applying \cite[Cor. 3.8]{abbottbehrstockdurham:largest} yields a $G$--HHS structure $(G,\mathfrak T)$ with unbounded products.  Examining the proof of the aforementioned theorem, each element $U\in\mathfrak T$ is of one of three types.  

First, $U$ can be a \emph{dummy domain} coming from the application of \cite[Thm. A.1]{abbottbehrstockdurham:largest}, and in this case $\mathcal CU$ is a single point.

Second, $U$ can correspond to an element of $\mathfrak S$, with stabiliser $\Stab_G(U)$ and associated hyperbolic space $\mathcal CU$ unchanged from the original HHG/HHS structure.

Third, we could have $U=S$.  In this case, $\mathcal CS$ is replaced by a hyperbolic space $\mathcal TS$ which is an electrification of $G$, so there is a coarsely lipschitz map $G\to\mathcal TS$, and the construction makes the original map $\pi_S:G\to\mathcal CS$ factor as $G\to \mathcal TS\to\mathcal CS$, where the latter map is also coarsely lipschitz.  So $\mathcal TS$ is unbounded if $\mathcal CS$ was.  All of the maps involved are $G$--equivariant, and the action of $G$ on $\mathcal TS$ is cobounded since the latter is an equivariant electrification.  We take $T=S$ and $\mathcal CT=\mathcal TS$ in the new HHG structure.

Thus $(G,\mathfrak T)$ has unbounded products and is therefore a well-behaved HHG structure, with unbounded maximal hyperbolic space by Lemma \ref{lem:irreducible}.
\end{proof}

Now we can state and prove the HHG part of Theorem \ref{thmint:acyl} from the introduction, along with a slightly more general version:

\begin{theorem}[AH from QI to $G$--HHS]\label{thm:HHG-QI-acyl}
Let $G$ be a nonelementary finitely generated group that is not quasi-isometric to the product of two unbounded spaces.  Suppose that $G$ admits a reasonably behaved $G$--HHS structure.  Then any finitely-generated group $H$ quasi-isometric to $G$ is acylindrically hyperbolic.
\end{theorem}

\begin{proof}
First suppose that $(G,\s)$ is a reasonably behaved $G$--HHS structure.  Using Lemma \ref{lem:making-good-behaviour} and Lemma \ref{lem:irreducible}, we can assume that $(G,\s)$ is well behaved and irreducible.  Let $S\in\s$ be the maximal element.  By \cite[Thm. 14.3]{behrstockhagensisto:hierarchically:1}, $G$ acts on $\mathcal CS$ acylindrically.  By Theorem \ref{thm:qi-descend}, along with the well-behaved assumption, quasi-isometries of $G$ descend through $\pi_S$ to $\mathcal CS$ (since, by Definition \ref{defn:HHG}, $\pi_S$ is an orbit map), in the sense of Definition \ref{defn:strong-21}.\eqref{item:descend-strong21}. Finally, equivalence of Morseness of geodesics in $G$ and quasi-geodesicity of their compositions with $\pi_S$ is given by \cite[Cor. 6.2]{abbottbehrstockdurham:largest}.  Hence $(G,\mathcal CS)$ is a \needaname pair as in Definition \ref{defn:strong-21}.  By Corollary \ref{cor:acyl-qi}, any $H$ as in the statement is acylindrically hyperbolic.
\end{proof}

\begin{remark}\label{rem:HHG-weak-21}
The proof of Theorem \ref{thm:HHG-QI-acyl} shows that any group $G$ that admits a reasonably behaved irreducible HHG structure $(G,\mathfrak S)$  admits an action on a hyperbolic space $X$ such that $(G,X)$ is a \needaname pair, and in particular, $(G,X)$ satisfies Assumption \ref{assumpt:space}.  If the original HHG structure had unbounded products, then we can take $X=\C S$.  And, moreover, the preceding holds in the slightly more general setting of a $G$--HHS, not just an HHG.
\end{remark}

\appendix
\renewcommand{\sectionname}{}

\section{Appendix: Quasi-isometries of HHGs with one-ended maximal \mbox{hyperbolic spaces}, by Jacob Russell}
\label{sec:appendix}

In this appendix, we prove a converse of Theorem \ref{thm:qi-descend} when $(X,\s)$ is a hierarchically hyperbolic group and the $\nest$-maximal space $\C S$ is one-ended.

\begin{thm}\label{thm:one-ende_converse}
	Let $(G,\s)$ be a hierarchically hyperbolic group with unbounded products. If $\C S$ is one-ended, then every quasi-isometry $f$ of $\C S$ induces a quasi-isometry $F$ of $G$ so that $f \circ \pi_S$ and $\pi_S \circ F$ coarsely agree. In particular, if $(G,\s)$ is well-behaved as described in Definition \ref{defn:well-behaved-HHS}, then $f \to F$ induces a group isomorphism $\operatorname{QI}(\C S) \to \operatorname{QI}(G)$.
\end{thm}
\begin{remark}
	We prove Theorem \ref{thm:one-ende_converse} under looser hypotheses than an HHG. All one needs is that the Cayley graph has an HHS structure on which the group acts by automorphisms. See Section \ref{sec:induced_maps}.
\end{remark}

Our proof of Theorem \ref{thm:one-ende_converse} uses the machinery of quasi-m\"obius maps on the Morse boundary developed by  Charney, Cordes, and Murray \cite{charneycordesmurray:quasi} and independently by Mousley and Russell in the case of HHGs \cite{mousleyrussell:quasi-mobius}.  The idea is that the quasi-isometry $f$ of $\C S$ will induce a quasi-m\"obius map on the boundary $\partial \C S$. Using a result of Abbott, Behrstock, and Durham about unbounded products, this quasi-m\"obius map on  $\partial \C S$ can be upgraded to a quasi-m\"obius map on the Morse boundary of $G$. The results of Charney, Cordes, and Murray (or Mousley and Russell in the HHG case), then say that this quasi-m\"obius map on the Morse boundary is induced by  a quasi-isometry of the group.

One-endedness of $\C S$ comes into play in being able to upgrade the map on $\partial \C S$ to a map on the Morse boundary. Being one-ended allows us to adopt an idea of Rafi and Schleimer for the curve graph and the mapping class group  \cite{rafischleimer_rigid}. 

Theorem \ref{thm:one-ende_converse} can fail when $\C S$ is not one-ended. For example, both the the fundamental group of a non-geometric graph manifold and any non-relatively hyperbolic right angled Artin group have HHG structures with unbounded products and where the $\nest$-maximal hyperbolic space is a quasi-tree.  However, these examples are known to not all be quasi-isometric \cite{behrstockneumann:graph_manifolds}.

\subsection{The Morse Boundary}
We briefly recall the relevant properties of the Morse boundary that we shall need. We direct the reader to \cite{cordes:morse} for a more detailed account.

Let $X$ be a proper geodesics metric space and fix a basepoint $x_0 \in X$. As a set, the Morse boundary of $X$ is the collection of all Morse geodesic rays based at $x_0$ up to asymptotic equivalence. We will topologize this boundary with the topology introduced by Cordes and denote this topological space by $\partial_\ast X_{x_0}$; see \cite{cordes:morse} for details. The primary fact we need about this topology is that any quasi-isometry $f \colon X \to Y$ of proper geodesic metric spaces has a continuous extension to a homeomorphism $\partial f \colon \partial_\ast X_{x_0} \to \partial_\ast Y_{f(x_0)}$. This shows that the boundary is not affected by the choice of basepoint. Moreover, there is a well defined Morse boundary for a finitely generated group $G$ by identifying it with the Morse boundary of any proper geodesic space on which it acts geometrically---usually its Cayley graph. We denote the Morse boundary of $G$ by $\partial_\ast G$ and will always assume it is identified with the Morse boundary of some finitely generated Cayley graph. The Morse boundary is visual in the sense that for any two points $p,q \in \partial_\ast X_{x_0}$, there is some bi-infinite Morse geodesic between $p$ and $q$ (the specific Morse gauge will depend on $p$ and $q$).

\subsection{Cross-ratio in hyperbolic spaces}
Let $X$ be a (not necessarily proper) $\delta$-hyperbolic space. Every pair of distinct points $a,b \in \partial X$ is joined by a bi-infinite $(1,20\delta)$-quasi-geodesic. A point $x \in X$ is a \emph{$K$-center} for the triple $(a,b,c) \in (\partial X)^3$  if $x$ is within $K$ of any  $(1,20\delta)$-quasi-geodesic between any two of $a,b,c$. There exists $K_\delta$, so that for every triple $(a,b,c) \in (\partial X)^3$ the set of $K_\delta$-centers is non-empty provided $a,b,c$ are all distinct. Let $m(a,b,c)$ be the set of $K_\delta$-centers for $(a,b,c)$.  There exists $D=D(\delta)$ so that $\diam(m(a,b,c)) \leq D$ whenever $m(a,b,c)$ is non-empty.

For distinct points $a,b,c,d \in \partial X$ define the \emph{cross-ratio} $[a,b,c,d]$ to be 
\[[a,b,c,d] := \diam(m(a,b,c) \cup m(a,d,c)).\]
This cross ratio is an additive error away from the absolute value of the cross ratio defined by Paulin; see \cite[Lemma 4.2]{paulin:quasi-mobius} for the proper case and \cite[Proposition 4]{mousleyrussell:quasi-mobius} for the non-proper case.

The next lemma is proved by Paulin when $X$ is proper \cite{paulin:quasi-mobius}. The same proof  works in the non-proper case if you replace geodesics with $(1,20\delta)$-quasi-geodesics.

\begin{lemma}\label{lem:hyp_qi_cross_ration}
	Let $X$ and $Y$ be $\delta$-hyperbolic spaces where $\partial X$ has at least 4 points. If $f \colon X \to Y$ is a $(\lambda,\epsilon)$-quasi-isometric embedding, then there exist $\lambda' \geq 1$ and $\epsilon' \geq 0$ determined by $\lambda, \epsilon,$ and $\delta$ so that \[ [\partial f(a), \partial f(b), \partial f(c), \partial f(d)]  \leq \lambda' [a,b,c,d] + \epsilon'\] for all $a,b,c,d \in \partial X$.
\end{lemma}

\subsection{Cross-ratio on the Morse boundary}

Let $G$ be a finitely generated group and $X$ a finitely generated Cayley graph. Let $\partial_\ast G$ denote the Morse boundary of a group identified with $\partial_\ast X$. For $k \in \mathbb{N}$, let $\partial^{(k,M)}_\ast G$ denote $k$-tuples of distinct elements of the Morse boundary so that every pair of points in the $k$-tuple are joined by a bi-infinite $M$-Morse geodesic in $X$ (this set is $G$-invariant). The next set of definitions are from \cite{charneycordesmurray:quasi} or \cite{mousleyrussell:quasi-mobius}.

Let $H$ be a second finitely generated group. A map $h \colon \partial_\ast G \to \partial_\ast H$ is \emph{2-stable} if for each Morse $M$ there is a Morse gauge $M'$ so that $h(\partial_\ast^{(2,M)} G) \subseteq \partial_\ast^{(2,M')} H$. 

For each triple $(a,b,c) \in \partial_\ast^{(3,M)} G$ a point $x \in G$ is a \emph{$K$-center} for $(a,b,c)$ if $x$ is within $K$ of all three sides of any $M$-Morse ideal triangle with endpoints $a,b,c$. For each $M$, there is number $K_M$ so that for any $(a,b,c) \in \partial_\ast^{(3,M)} G$ 
the set of $K_M$-centers is non-empty. Moreover, there exist $D = D(M)$ so that for each $(a,b,c) \in \partial_\ast^{(3,M)} G$, the set of $K_M$-centers, $m(a,b,c)$, has diameter at most $D$. For any tuple $(a,b,c,d) \in \partial_\ast^{(4,M)} G$, the \emph{$M$-cross ratio} is $$ [a,b,c,d]_M := \diam(m(a,b,c) \cup m(a,d,c)).$$ A $2$-stable map $h\colon \partial_\ast G \to \partial_\ast H$ is \emph{quasi-m\"obius} if for every pair of Morse gauges $M$ and $M'$ with  $h(\partial_\ast^{(2,M)} G) \subseteq \partial_\ast^{(2,M')} H$ there exists an increasing function $\psi$ so that $$ [h(a),h(b),h(c),h(d)]_{M'} \leq \psi([a,b,c,d]_M).$$

Charney, Cordes, and Murray established that quasi-m\"obius maps on the Morse boundary characterize quasi-isometries of the group.

\begin{thm}[{\cite[Theorems 3.7 and 4.5] {charneycordesmurray:quasi}}]\label{thm:quasi-mobius = quasi-isometry} 
	Let $G$ and $H$ be finitely generated groups so that $\partial_\ast G$ contains at least 3 points. 
	\begin{enumerate}
		\item If $F\colon G \to H$ is a quasi-isometry, then the induced map $\partial F \colon \partial_\ast G \to \partial_\ast H$ is a quasi-m\"obius homeomorphism with quasi-m\"obius inverse.
		\item If $h \colon \partial_\ast G \to \partial_\ast H$ is a quasi-m\"obius homeomorphism with quasi-m\"obius inverse, then there exist a quasi-isometry $F \colon G \to H$ so that $\partial F  = h$.
	\end{enumerate}
\end{thm}

\subsection{The Morse boundary  and unbounded products}
For the remainder of the appendix, we fix a finitely generated group $G$  and a finitely generated Cayley graph, $X$, of  $G$. We assume $X$ has an HHS structure $\s$ with unbounded products and let $S \in \s$ be the $\nest$-maximal domain. As described in \cite[Section 3]{abbottbehrstockdurham:largest}, we can assume that $\C S$ is a graph obtained from $X$ by adding additional edges between the vertices. In this case, $\pi_S \colon X \to \C S$ is taken to be  the inclusion map.

Abbott, Behrstock, and Durham showed that in the presence of unbounded products, Morse geodesic are characterized by projecting  to  (parameterized) quasi-geodesics in $\C S$. 

\begin{thm}[{\cite[Corollary 6.2]{abbottbehrstockdurham:largest}}]\label{thm:ABD_Morse}
	Let $(X,\s)$ be an HHS with unbounded product and $S\in\s$ be the $\nest$-maximal domain. Let $\gamma$ be a geodesic in $X$. 
	\begin{enumerate}
		\item If $\pi_S \circ \gamma$ is a parameterized $(\lambda,\lambda)$-quasi-geodesic in $\C S$, then $\gamma$ is $M$-Morse for some Morse gauge $M$ determined by $\lambda$ and $\s$.
		\item If $\gamma$ is $M$-Morse, then there exists $\lambda = \lambda(M,\s) \geq 1$ so that $\pi_S \circ \gamma$ is a parameterized $(\lambda,\lambda)$-quasi-geodesic in $\C S$.
	\end{enumerate}
\end{thm}

Theorem \ref{thm:ABD_Morse} implies there is a continuous injection $\partial \pi_S \colon \partial_\ast X \to \partial \C S$, which is a continuous extension of $\pi_S \colon X \to \C S$, see \cite[Lemma A.6]{russell:extensions}.

\subsection{Downward relative projections and coboundedness}
For each $W \in \s -\{S\}$ define $\rho_W^S \colon \C S^{(0)} \to \C W $ by $\rho_W^S(v) = \pi_W \circ \pi_S^{-1}(v)$. To extend this to points in $\partial \C S$ we need the next lemma, which is a basic consequence of the bounded geodesic image axiom of an HHS.

\begin{lemma}\label{lem:BGI_quasi-geodesic}
	Let $(X,\s)$ be an HHS with hierarchy constant $E$ and $S$ be the $\nest$-maximal domain of $\s$. For each $\kappa \geq 1$ there exist $\nu = \nu (\kappa,E)$ so that for all $W \in \s -\{S\}$ we have:
	\begin{enumerate}
		\item if $\gamma$ is a $(\kappa,\kappa)$-quasi-geodesic in  $\C W$ so that $\gamma \cap \cal N_\nu(\rho_S^W) = \emptyset$, then $$\diam(\rho_W^S(\gamma)) \leq E.$$
		\item  if $\gamma_1$ and $\gamma_2$ are  $(\kappa,\kappa)$-quasi-geodesic rays in  $\C S$ that both represents the same point $p \in \partial \C S$ and $\gamma_i \cap \cal N_\nu(\rho_S^W) = \emptyset$ for $i = 1$ and $2$, then $d_{Haus}(\rho_W^S(\gamma_1), \rho_W^S (\gamma_2)) \leq \nu.$
	\end{enumerate}
\end{lemma}

Now for each $p \in \partial \C S$, define $\partial \rho_W^S$ as follows: let $Z$ be the set of $(1,20E)$-quasi-geodesics from the basepoint of $\partial \C S$ to $p$. Let $\nu = \nu(E)$ be the constant from Lemma \ref{lem:BGI_quasi-geodesic} for $\kappa = 20E$. For $W \in \s -\{S \}$, let $Z_W$ be the subset of $Z$ that is at least $\nu$ far from $\rho_S^W$. Define $\partial \rho_W^S (p)$ to be the $\nu$-bounded diameter set $\rho_W^S(Z_W)$.

If $p$ is a point in the Morse boundary $ \partial_\ast X$, we define $\partial \pi_W(p)$ as $\partial \rho_W^S(\partial \pi_S(p))$.

If $x,y$ are are points in any combination of $X$, $\partial_\ast X$, $\C S$, or $\partial \C S$,  we say $x$ and $y$ are \emph{$C$-cobounded} if the union of their projections to $\C W$ for each $W \in \s -\{S\}$ has diameter at most $C$. Here the projection is under $\pi_W$, $\partial \pi_W$, $\rho_W^S$, or $\partial \rho_W^S$ depending on which space $x$ and $y$ are in.

In the language of coboundedness, Corollary 6.2 of \cite{abbottbehrstockdurham:largest} becomes the following.

\begin{thm}[Restatement of {\cite[Corollary 6.2]{abbottbehrstockdurham:largest}}]\label{thm:cobounded=Morse}
	Let $(X, \s)$ be a proper HHS with unbounded products and hierarchy constant $E$. Let $S \in \s$ be the $\nest$--maximal domains and $x,y \in X \cup \partial_\ast X$.
	\begin{enumerate}
		\item For all Morse gauges $M$, there exist a constant $C \geq 0$, depending on $M$ and $E$, so that if $x$ and $y$ are joined by an $M$-Morse geodesic, then $x$ and $y$ are $C$-bounded.
		\item For all $C \geq 0$, there exists a Morse gauge $M$, depending on $C$ and $E$, so that if $x$ and $y$ are $C$-cobounded, then there exists an $M$-Morse geodesic from $x$ to $y$. 
		\item For all $C \geq 0 $ and $p \in \partial \C S$, if $\pi_S(x)$ and $p$ are $C$-cobounded, then there exists a Morse gauge $M$, depending on $C$ and $E$, and $z \in \partial_\ast X$ so that $\partial \pi_S (z) = p$ and the geodesic from $x$ to $z$ is $M$-Morse.
	\end{enumerate}
\end{thm}

\subsection{Quasi-m\"obius maps induced by quasi-isometry of $\C S$}	\label{sec:induced_maps}
For our final section, we require that $G$ is a group of \emph{automorphisms} of the HHS  $(X,\s)$.  This means that $G$ acts on $\s$ by $\nest$-, $\perp$-, and $\trans$-preserving bijections, and for each  $W \in \s$ and $g\in G$, there exists an isometry $g_W \colon \C W \rightarrow \C gW$ satisfying the following for all $V,W \in \s$ and $g,h \in G$.
\begin{itemize}
	\item The map $(gh)_W \colon \C W  \to \C ghW$ is equal to the map $g_{hW} \circ h_W \colon \C W \to \C ghW$.
	\item For each $x \in X$, $g_W(\pi_W(x))$ uniformly coarsely equals  $\pi_{gW}(g \cdot x)$.
	\item If $V \trans W$ or $V \propnest W$, then $g_W(\rho_W^V)$ uniformly coarsely equals $\rho_{gW}^{gV}$.
\end{itemize}
This equivariance hypothesis is required to establish the next lemma, which is the important technical step in establishing that quasi-isometries of $\C S$ preserve being cobounded when $\C S$ is one-ended.

\begin{lemma}\label{lem:QI_preserve_coboundedness}
	Let $G$ be a finitely generated group and $X$ any Cayley graph with respect to a finite generating set. Let $\s$ be an HHS structure for $X$ with hierarchy constant $E$ and suppose $G$ is a group of automorphisms of $(X,\s)$. Let $S \in \s$ be the $\nest$--maximal domains and assume $\C S$ is one-ended. For all $C, R \in \mathbb{N}$, there exists $c,r \in \mathbb{N}$ so that the following holds: Suppose $x,y$ are elements of $G$ that are $C$-bounded, and let $z$ be the midpoint of a $M$-Morse geodesic from $x$ to $y$ in $X$, where $M$ is determined by $C$ and $E$ as in Theorem \ref{thm:cobounded=Morse}. If $d_X(x,y) = 2r$, then 
	\begin{enumerate}
		\item $\pi_S(x)$ and $\pi_S(y)$ are at least $R+1$ far from $\pi_S(z)$ and
		\item there is a path $\eta$ in $\C S$ of length less than $c$  that connects $\pi_S(x)$ to $\pi_S(y)$ in $\C S - B_R(\pi_S(z))$.
	\end{enumerate}  
\end{lemma}

\begin{proof}
	Let $x$ and $y$ be elements of $G$ that are $C$-bounded and $z$ be the  the midpoint of a $M$-Morse geodesic from $x$ to $y$, where $M$ is determined by $C$ and $E$ as in Theorem \ref{thm:cobounded=Morse}. Thus, $x,z$ and $y,z$ are $C'$-cobounded for some $C'$ depending on $C$. Hence the uniqueness axiom of an HHS implies that for each $R \geq 0$, there exists $r = r(R, E,C)$ so that $\dist_X(x,z) = \dist_X(y,z) \geq r$ implies $\dist_S(x,z),\dist_S(y,z) \geq R+1$. Since $\C S$ is one-ended, we can connect $\pi_S(x)$ to $\pi_S(y)$ by a path $\eta_{xy}$ that avoids $B_R(\pi_S(z))$.
	
	Now, since balls in $X$ contain finitely many elements of $G$, for each $r$, there only exists finitely many $G$-orbits of triples of $x,y,z$ so that $d_X(x,y)=2r$ and $x,y$ are joined by an $M$-Morse geodesic with midpoint $z$. Since $\pi_S$ is coarsely equivariant with respect to the the action of $G$, for each $R$ and $C$, we can pick $\eta_{xy}$ so that its length in $\C S$ is at most $c = c(R,C,E)$.
\end{proof}

We now fix two finitely generated groups $G$ and $H$, and let $X$ and $Y$, respectively, be Cayley graphs with respect to finite generating sets. We assume there are HHS structures $\s$ and $\mathfrak T$ for $X$ and $Y$ respectively so that $\s$ and $\mathfrak T$ have unbounded products and that $G$ and $H$ are groups of automorphism of $(X,\s)$ and $(Y,\mathfrak T)$ respectively. Let $S\in \s$ and $T \in \mathfrak T$ be the respective $\nest$-maximal domains, and let $E$ be the hierarchy constant for both $\s$ and $\mathfrak T$.

We now prove that when $\C S$ is one-ended, quasi-isometries $\C S \to \C T$ will preserve coboundedness and hence induce a quasi-m\"obius map on the Morse boundary.  Our proof of Theorem \ref{thm:qi_curve_graph_implies_qi_spaces} is inspired by work of Rafi and Schleimer in the case of the mapping class group \cite{rafischleimer_rigid}.

\begin{thm}\label{thm:qi_curve_graph_implies_qi_spaces}
	Suppose $f \colon \C S \to \C T$ is a $(\lambda,\epsilon)$-quasi-isometric embedding. Let $\partial f \colon \partial \C S \to \partial \C T$ be the topological embedding induced by $f$. If $\C S$ is one-ended, then: 
	\begin{enumerate}
		\item \label{item:induce_morse_map} 	There exists a topological embedding $h\colon \partial_* G \to \partial_* H$ so that $$\partial \pi_T \circ h(p) = \partial f \circ \partial \pi_S(p)$$ for all $p \in \partial_\ast G$.
		
		\item \label{item:2-stable} If $p,q \in \partial \C S$ are $C$-cobounded, then $\partial f(p)$ and $\partial f(q)$ are $C'$-cobounded for some $C' = C'(C,\lambda,\epsilon, E)$. In particular, the map $h$ from \eqref{item:induce_morse_map} is $2$-stable.
		
		\item \label{item:quasi-mobius} The map $h$  from \eqref{item:induce_morse_map} is quasi-m\"obius with increasing function $\psi$ determined by $\lambda$, $\epsilon$, and $E$.
	\end{enumerate}
\end{thm}

\begin{proof} 
	Without loss of generality, we can assume the image of $f$ is contained in the vertices of $\C T$. Let $e$ be the identity in $G$ and $b$ be the element of $H$ so that $\pi_T(b) = f \circ \pi_S(e)$. Let $e$ and $b$ be the base points for  the Morse boundaries of $G$ and $H$ respectively.
	
	\textbf{Proof of \eqref{item:induce_morse_map}:} Let $p \in \partial_\ast G$ and $p_S = \partial \pi_S(p)$. We will first show that $\partial f(p_S) = \partial \pi_T(q_p)$ for some point $q_p \in \partial_\ast G$.  By Theorem \ref{thm:cobounded=Morse}, it suffices to show that $\partial f(p_S)$ and $\pi_T(b)$ are cobounded (although not necessarily uniformly).
	
	Let $\gamma \colon [0,\infty) \to G$ be the $M$-Morse geodesic ray from $e$ to $p$ and let $W \in \frak T$. There exist $C = C(M,E)$ so that any two points on $\gamma$ are $C$-cobounded by Theorem \ref{thm:cobounded=Morse}. By Theorem \ref{thm:ABD_Morse},  $\bar{\gamma} = f \circ \pi_S \circ \gamma \colon [0,\infty) \to \C T$ is a $(\kappa,\kappa)$-quasi-geodesic for some $\kappa = \kappa( \lambda, \epsilon, M, E) \geq 1$.  Without loss of generality, we can assume $\kappa \geq \max\{\lambda,\epsilon,20E\}$. Let $\nu=\nu(E,\kappa) \geq 0$ be the constant from Lemma \ref{lem:BGI_quasi-geodesic} so that $\diam(\rho_W^T(\bar{\gamma})) \leq E$ whenever $\bar{\gamma}$ is not within  $\nu$ of $\rho_T^W$. Let $R = \lambda(2\nu + \epsilon +1)$ and let $r, c \in \mathbb{N}$ be as in Lemma \ref{lem:QI_preserve_coboundedness} for this choice of $R$ and our given $C$.   
	Hence, we can assume there exists $t\in [0,\infty) \cap \mathbb{Z}$ so that 
	$\dist_T(\bar{\gamma}(t),\rho_T^W) \leq \nu$. 
	
	To start, assume $t >2r +1$. By Lemma \ref{lem:QI_preserve_coboundedness}, there exists a path $\eta$ from $\pi_S \circ \gamma(t-r)$ to $\pi_S\circ\gamma (t+r)$ in $ \C S$  so that $\eta$ has length at most $c$ and $\eta$ does not intersect the ball of radius $R$ around $\pi_S \circ \gamma(t)$. 
	Now consider a single edge $\eta_0$ of $\eta$.  Since $d_S(\eta_0,\pi_S\circ \gamma(t)) >R= \lambda(2\nu + \epsilon +1)$ and $\bar{\gamma} = f \circ \pi_S \circ \gamma$, we have
	\begin{equation*}
		d_T(f(\eta_0), \rho_T^W) \geq d_T\left(f(\eta_0),\bar{\gamma}(t)\right) - d_T\left(\bar{\gamma}(t),\rho_T^W\right) > \frac{1}{\lambda} R - \epsilon - \nu = \nu+1. \tag{$\ast$} \label{eq:BGI_condition}
	\end{equation*}

	Since $\kappa \geq \max\{\lambda, \epsilon\}$, $f(\eta_0)$ is a $(\kappa,\kappa)$-quasi-geodesic in $\C T$, thus \eqref{eq:BGI_condition} implies that $\diam(\rho_W^S(f(\eta_0))) \leq E$. Since $\eta$ has length $c$, this implies $\diam(\rho_W^T(f (\eta)) \leq cE$. From the proof of Lemma \ref{lem:QI_preserve_coboundedness}, $r$ is chosen so that for any $s \in [0,\infty)$ we have  $$|s-t| \geq r \implies d_S(\pi_S \circ \gamma(s), \pi_S\circ \gamma(t) ) \geq R+1.$$ Thus, our choice of $R$ ensures that $\bar{\gamma}\vert_{[0,t-r]}$ and $\bar{\gamma}\vert_{[t+r,\infty)}$ do not intersect $\cal N_\nu(\rho_T^W)$. Thus, $\diam(\rho_W^T(\bar{\gamma}([0,t-r])))$ and $\diam(\rho_W^T(\bar{\gamma}([t+r, \infty))))$ are both at most $E$ by Lemma \ref{lem:BGI_quasi-geodesic}. Since $\diam(\rho_W^T(f (\eta)) \leq cE$, we have $\diam(\rho_W^T(\bar{\gamma}\vert_{[0,t-r]})) \cup \rho_W^T(\bar{\gamma}\vert_{[t+r, \infty)})) \leq 2E + cE$. 
	
	Now $\pi_W(b) \subseteq \rho_W^T(\bar{\gamma}\vert_{[0,t-r]})$ because  $\pi_W(b) = \rho_W^T( f\circ \pi_S(e))$. Since $\bar{\gamma}\vert_{[t+r,\infty)}$ represents $\partial f(p_S)$ and does not intersect the $\nu$-neighborhood of $\rho_W^T$, Lemma \ref{lem:BGI_quasi-geodesic} says $\partial \rho_W^S(\partial f(p_S))$ is $\nu$-close to $\rho_W^T(\bar{\gamma}\vert_{[t+r,\infty)})$. Hence $$\diam(\rho_W^T(\pi_T(b)) \cup \partial \rho_W^T(\partial f(p_S))) \leq 2E+ cE +\nu.$$
	
	Now assume $0\leq t \leq 2r+1$.  As in the previous case, our choice of $R$ and $r$ ensures that $\bar{\gamma}\vert_{[2r+1,\infty)}$ does not intersect $\cal N_\nu(\rho_T^W)$, which implies $\diam(\rho_W^T(\bar{\gamma}\vert_{[2r+1,\infty)}))\leq E$ by Lemma \ref{lem:BGI_quasi-geodesic}. Since $\rho_W^T \circ \pi_T = \pi_W$ for each $W \in \mathfrak T - \{T\}$, the distance formula for an HHS implies there is $A= A(E) \geq 0$ so that  set of possible domains where $\diam(\rho_W^T(\bar{\gamma}([0,2r+1])) \geq A$ is finite. Hence there is some bound $D \geq 0$ (depending on $p$ and $b$) on $\diam(\rho_W^T(\bar{\gamma}))$.  As in the previous case, this implies $$\diam(\rho_W^T(\pi_T(b)) \cup \partial \rho_W^T(\partial f(p_S))) \leq D +\nu.$$
	
	Taking $C' = \max\{D+\nu, 2E + cE + \nu\} $, the above shows that $\pi_T(b)$ and $\partial f(p_S)$ are $C'$-bounded. Hence, by Theorem \ref{thm:cobounded=Morse}, there exists $q_p \in \partial_\ast H$ so that $\partial \pi_T(q_p) = \partial f( \partial_S(p))$.
	
	We now show that the map  $h\colon \partial_\ast G \to \partial_\ast H$ defined by $h(p) = q_p$ is a topological embedding. By construction $ \partial \pi_T \circ h= \partial f \circ \partial \pi_S$. Since $\partial f$, $\partial \pi_S$, and $\partial \pi_T$ are all topological embeddings, $h$ must also be a topological embedding.
	
	\textbf{Proof of \eqref{item:2-stable}:} Let $p,q \in \partial \C S$ be $C$-cobounded. By Theorem \ref{thm:cobounded=Morse}, there exist $p',q' \in \partial_\ast G$ so that $\partial \pi_S(p') = p$ and $\partial \pi_S(q') = q$. We will show $f(p)$ and $f(q)$ are $C'$-cobounded for some $C'$ depending on $\lambda$, $\epsilon$, $C$, and $E$. 
 
 By Theorem \ref{thm:cobounded=Morse}, there exists a Morse gauge $M = M(C,E)$ so that $p'$ and $q'$ are connected by a bi-infinite $M$-Morse geodesic in $G$. Let $\gamma \colon (-\infty, \infty) \to G$ be such an $M$-Morse geodesic and let $\bar{\gamma} = f\circ \pi_S \circ \gamma$. Let $W \in \mathfrak T - \{T\}$. By Theorem \ref{thm:ABD_Morse}, $\pi_S \circ \gamma$ is a $(\kappa,\kappa)$-quasi-geodesic in $\C S$ for $\kappa = \kappa(M,E)$.  Lemma \ref{lem:BGI_quasi-geodesic}  says there is $\nu = \nu(\lambda, \epsilon,\kappa, E)$ so that if $\cal N_\nu(\rho_T^W) \cap \bar{\gamma} = \emptyset$, then  $\diam(\rho_W^T( \bar{\gamma})) \leq E$. Hence, we can assume there is $t \in \mathbb{Z}$ so that $\bar{\gamma}(t) \in \cal N_\nu(\rho_T^W)$.  Let $R = \lambda(2\nu + \epsilon + 1)$ and let $r, c \in \mathbb{N}$ be as in Lemma \ref{lem:QI_preserve_coboundedness} for this choice of $R$ and our given $C$. Thus, there exists a path $\eta$ in $\C S$ connecting $\pi_S \circ \gamma (t-r)$ and $\pi_S \circ \gamma (t+r)$ so that $\eta$ avoids the $R$-ball around $\pi_S\circ \gamma (t)$ and $\eta$ has length at most $c$. 
 
 By arguing exactly as we did in the $t > 2r+1$ case of the proof of \eqref{item:induce_morse_map}, we have $\diam(\rho_W^T (\eta)) \leq cE$ plus  $$\diam(\rho_W^T(\bar{\gamma}\vert_{(-\infty,r-t]})) \leq E \text{ and } \diam(\rho_W^T(\bar{\gamma}\vert_{[r+t,\infty)}))\leq E.$$ Continuing as in the proof of \eqref{item:induce_morse_map}, this will imply $\partial f(p)$ and $\partial f(q)$ are $ (2E + cE + 2\nu)$-cobounded.
	
	The fact that the homeomorphism $h$ from \eqref{item:induce_morse_map} is 2-stable now follows due to the  correspondence in Theorem \ref{thm:cobounded=Morse} between being cobounded and being joined by a Morse geodesic.
	
	\textbf{Proof of \eqref{item:quasi-mobius}:}  For distinct $a,b,c,d \in \partial \C S$, let $[a,b,c,d]_S$ denote the cross-ratio in $\partial \C S$. By Lemma \ref{lem:hyp_qi_cross_ration}, $f$  being a quasi-isometric embedding implies $\partial f$ is quasi-m\"obius with respect to the cross ratio on $\partial \C S$. Proposition 7 of \cite{mousleyrussell:quasi-mobius} shows that for each Morse gauge $M$ and each tuple $(a,b,c,d) \in \partial_\ast^{(4,M)} G$, we have \[ [a,b,c,d]_M \asymp [\partial \pi_S(a), \partial \pi_S(b), \partial \pi_S(c), \partial \pi_S(d)]_S\] with constants determined by $E$ and $M$. Hence the map $h$ from \eqref{item:induce_morse_map} is quasi-m\"obius. 
\end{proof}

Our main result (Theorem \ref{thm:one-ende_converse}) is now a corollary of Theorem \ref{thm:qi_curve_graph_implies_qi_spaces}.

\begin{corollary}
	If $\C S$ is one-ended and there exists a $(\lambda,\epsilon)$-quasi-isometry $f \colon \C S \to \C T$, then there exists a quasi-isometry $F \colon G \to H$ so that for all $x \in G$, $\pi_T \circ F(x)$ is uniformly close to $f \circ \pi_S(x)$.
\end{corollary}

\begin{proof}
	Theorem 14.3 of \cite{behrstockhagensisto:hierarchically:1} says $G$ acts cobounded and acylindrically on  $\C S$. Thus, $\C S$ being one-ended  implies that $\partial_\ast G$ contains at least 3 points.
	
	Let $f^{-1} \colon \C T \to \C S$ be a quasi-inverse for $f$. Applying Theorem \ref{thm:qi_curve_graph_implies_qi_spaces} to both $f$ and $f^{-1}$, we produce a quasi-m\"obius homeomorphism $h \colon \partial_\ast G \to \partial_\ast H$ with quasi-m\"obius inverse so that $\partial f \circ \partial \pi_S = \partial \pi_T \circ h$. By Theorem \ref{thm:quasi-mobius = quasi-isometry} there is a quasi-isometry $F \colon G \to H$ so that $\partial F = h$. 
 
    Because $G$ acts cocompactly on $X$, there is a Morse gauge $M$  and a constant $K_M \geq0$ so that for all $x \in G$ there is a triple $(p,q,z)\in \partial^{(3,M)}_\ast G$ so that $x$ is a $K_M$-center for $(p,q,z)$. Now $(\partial F(p),\partial F(q), \partial F(z)) \in \partial_\ast^{(3,M')} H$ for some $M'$ determined by $M$, $\lambda$, $\epsilon$, and $E$. Moreover, $F(x)$ is uniformly close to all three sides of any ideal $M'$-Morse triangle with vertices $\partial F(p), \partial F(q), \partial F(z)$. Since $\partial f \circ \partial \pi_S = \partial \pi_T \circ \partial F$, this implies $\pi_T \circ F(x)$ is uniformly close to $f \circ \pi_S(x)$. 
\end{proof}

\begin{remark}
    There is a proof of Theorem \ref{thm:one-ende_converse} that does not rely on the Morse boundary and quasi-M\"obius maps, but instead directly invokes the quasi-isometry on the $\nest$-maximal hyperbolic space $\C S$.  This argument uses the fact that quasi-isometries of $\C S$ preserve the set of cobounded pairs in $\partial \C S$ to show that quasi-isometries of $\C S$ produce ``quasi-M\"obius'' maps on the set of cobounded tuples in $\partial \C S$. Thus, one can essentially repeat the  arguments used in  \cite{charneycordesmurray:quasi} to build a quasi-isometry of $G$  with coboundedness and the distance formula replacing the role of Morse geodesics. As this approach would result in a lengthier proof without any fundamentally new ideas,  we have elected to give the shorter proof using the established results from the literature.
\end{remark}


\bibliography{main}
\bibliographystyle{alpha}

\end{document}